\documentclass[a4paper]{article}
\usepackage{hyperref}
\usepackage{graphicx}
\usepackage{mathrsfs}
\usepackage{enumerate} 
\usepackage{amsxtra,amssymb,latexsym, amscd,amsthm}
\usepackage{indentfirst}
\usepackage{color}
\usepackage[utf8]{inputenc}
\usepackage[mathscr]{eucal}
\usepackage{amsfonts}
\usepackage{graphics}
\usepackage{multirow}
\usepackage{array}
\usepackage{subfigure}
\usepackage{cite}
\usepackage{wrapfig}
\usepackage{tikz}

\usepackage{pgfplots}
\pgfplotsset{compat=1.15}
\usepackage{mathrsfs}
\usetikzlibrary{arrows}

\newcommand{\footremember}[2]{%
    \footnote{#2}
    \newcounter{#1}
    \setcounter{#1}{\value{footnote}}%
}
\newcommand{\footrecall}[1]{%
    \footnotemark[\value{#1}]%
} 

\def\df{d_f}

\def\R{{\mathbb R}}
\def\N{{\mathbb N}}

\DeclareMathOperator{\dist}{dist}

\newtheorem{theorem}{\bf Theorem}
\newtheorem{lemma}{\bf Lemma}
\newtheorem{corollary}{\bf Corollary}
\newtheorem{definition}{\bf Definition}
\newtheorem{remark}{\bf Remark}

\providecommand{\keywords}[1]
{
  \small	
  \textbf{\textbf{Keywords:}} #1
}
\begin{document}
\definecolor{qqzzff}{rgb}{0,0.6,1}
\definecolor{ududff}{rgb}{0.30196078431372547,0.30196078431372547,1}
\definecolor{xdxdff}{rgb}{0.49019607843137253,0.49019607843137253,1}
\definecolor{ffzzqq}{rgb}{1,0.6,0}
\definecolor{qqzzqq}{rgb}{0,0.6,0}
\definecolor{ffqqqq}{rgb}{1,0,0}
\definecolor{uuuuuu}{rgb}{0.26666666666666666,0.26666666666666666,0.26666666666666666}
\newcommand{\vi}[1]{\textcolor{blue}{#1}}
\newif\ifcomment
\commentfalse
\commenttrue
\newcommand{\comment}[3]{%
\ifcomment%
	{\color{#1}\bfseries\sffamily#3%
	}%
	\marginpar{\textcolor{#1}{\hspace{3em}\bfseries\sffamily #2}}%
	\else%
	\fi%
}
\newcommand{\victor}[1]{
	\comment{blue}{V}{#1}
}

\newcommand{\hoang}[1]{
	\comment{red}{H}{#1}
}
\if{
\newcommand{\mapr}[1]{
	\comment{blue}{}{#1}
}
}\fi
\newcommand{\mapr}[1]{{{\color{blue}#1}}}
\newcommand{\revise}[1]{{{\color{blue}#1}}}

\title{
On the complexity of Putinar-Vasilescu's Positivstellensatz
}
\author{%
Ngoc Hoang Anh Mai\footremember{1}{CNRS; LAAS; 7 avenue du Colonel Roche, F-31400 Toulouse; France.}, %
  Victor Magron\footrecall{1} \footremember{2}{Universit\'e de Toulouse; LAAS; F-31400 Toulouse, France.} %
  }
\maketitle
\begin{abstract}
We provide a new degree bound on the weighted sum-of-squares (SOS) polynomials for Putinar-Vasilescu’s Positivstellensatz.
This leads to another Positivstellensatz saying that
if $f$ is a polynomial of degree at most $2\df$  nonnegative on a semialgebraic set having nonempty interior defined by finitely many polynomial inequalities $g_j(x)\ge 0$, $j=1,\dots,m$ with $g_1:=L-\|x\|_2^2$ for some $L>0$, then there exist positive constants $\bar c$ and $c$ depending on $f,g_j$ such that for any $\varepsilon>0$, for all $k\ge \bar c\varepsilon^{-c}$, $f$ has the decomposition
\begin{equation}
\begin{array}{l}
(1+\|x\|_2^2)^k(f+\varepsilon)=\sigma_0+\sum_{j=1}^m \sigma_jg_j \,,
\end{array}
\end{equation}
for some SOS polynomials $\sigma_j$ being such that the degrees of $\sigma_0,\sigma_jg_j$ are at most $2(\df+k)$. 
Here $\|\cdot\|_2$ denotes the $\ell_2$ vector norm.
As a consequence, we obtain  a converging hierarchy of semidefinite relaxations for lower bounds in polynomial optimization on basic compact semialgebraic sets.
The complexity of this hierarchy is $\mathcal{O}(\varepsilon^{-c})$ for prescribed accuracy $\varepsilon>0$.
In particular, if $m=L=1$ then $c=65$, yielding the  complexity $\mathcal{O}(\varepsilon^{-65})$ for the minimization of a polynomial on the unit ball.
Our result improves the complexity bound  $\mathcal{O}(\exp(\varepsilon^{-c}))$ due to Nie and Schweighofer in [Journal of Complexity 23.1 (2007): 135-150].
\end{abstract}
\tableofcontents
\section{Introduction}
\if{
{\color{red} Hoang \& Victor: 

The paper is nice but in my opinion needs a serious rewriting :
\begin{itemize}
    \item Put Background first before "contribution" with the two main
    convergence results before the present paper. 
    
    - I Schweighoffer's result on Schm\"udgen's representation (usnibg preordering with $2^m$ terms). 
    
    -II Nie \& Schweighofer's results on Putinar (using quadratic module with only $m$ terms) with exponential rate.
    
    - III Say that for the special case of the sphere, nicer results have been obtained by Fawzi (and earlier by Werher) and also LMaurebnt and collaborators for the hierarchy of upper bounds on specific sets.
    
    - IV  So the present result is concerned with the general case of compact basis semi-algebraic sets. Basically  one obtains a rate similar in spirit and magnitude to I, but still based on the quadratic module (and not the preordering) but now with a prescribed denominator (using Putinar's Vasilescu's Positivstellensatz.
    
    \item Then describe the main contribution in the "Contribution section"
    \item say a  few words on the basic ingredients of the technique used.
    \item By the way, Theorem 2 has been already provided by 
    Aquistapace et al. in a very general setting\footnote{F. Acquistapace, C. Andradas, F. Broglia. The Positivstellensatz for definable functions on O-minimal structures, \emph{Illinois J. Math. 46}, No 3, pp. 685--693, 2002.} and also for analytic functions in another paper. So I don't know if it is appropriate to state Theorem 2 if it is contained in the reference that I mentiojn in footnote
    \end{itemize}}
}\fi    
For a positive $m \in \N$, let us consider the polynomial optimization problem (POP):
\begin{equation}\label{eq:POP-def}
f^\star:=\inf_{x \in S} f( x)\,,
\end{equation}
where $f\in\R[x]$ and
\begin{equation}\label{eq:semial.set.def}
    S:=\{x\in\R^n\,:\,g_j(x)\ge 0\,,\,j\in[m]\}\,,
\end{equation}
for some $g_j\in\R[x]$, $j\in[m]$.
Here $\R[x]$ denotes the ring of real polynomials in vector of variables $x=(x_1,\dots,x_n)$ and $[m]$ stands for the set $\{1,\dots,m\}$.
Assume that $f$ has degree at most $2\df$ for some positive  $\df\in\N$.
The set $S$ is a conjunction of finitely many polynomial inequalities, and therefore is called a \emph{basic semialgebraic set}.

Problem \eqref{eq:POP-def} can be written as 
\begin{equation}\label{eq:equi.prob}
\begin{array}{l}
f^\star=\sup_{\lambda\in\R}\{\lambda\,:\,f-\lambda > 0\text{ on }S\}\,.
\end{array}
\end{equation}
We can replace the inequality constraint  of problem \eqref{eq:equi.prob} by an equality constraint, if one can represent positive polynomials on $S$.
%
Assume that $S$ has nonempty interior and a ball constraint is present, i.e., $g_1=L-\|x\|_2^2$ for some $L>0$. 
Our first (minor) contribution is to rely on the representation of polynomials positive on $S$ stated by Putinar and Vasilescu \cite{putinar1999positive}, to obtain
\begin{equation}\label{eq:rep.Pu-Va}
    f-\lambda=\frac{\sigma_0+\sum_{j\in[m]}\sigma_jg_j}{(1+\|x\|_2^2)^k}\,,
\end{equation}
for some $k\in\N$, $\sigma_j\in \Sigma[x]$, $j\in[m]$, being such that $\deg(\sigma_0)\le 2(k+\df)$ and $\deg(\sigma_jg_j)\le 2(k+\df)$.
Here $\Sigma[x]$ denotes the set of sum-of-squares (SOS) polynomials and $\deg(\cdot)$ stands for the degree of a polynomial.
Such a representation of positive polynomials is called a  \emph{Positivstellensatz}.

After bounding the degrees of the SOS polynomials involved in  \eqref{eq:rep.Pu-Va}, we obtain the following hierarchy of relaxations indexed by $k\in\N$:
\begin{equation}\label{eq:hier.puVa}
\begin{array}{rl}
     \rho^{(0)}_k:=\sup\limits_{\lambda,\sigma_j}& \lambda \\
     \text{s.t.}& \lambda\in\R\,,\,\sigma_j\in \Sigma[x]\,,\\
     &(1+\|x\|_2^2)^k(f-\lambda)=\sigma_0+\sum_{j\in[m]}\sigma_jg_j \,,\\
     &\deg(\sigma_0)\le 2(k+\df)\,,\,\deg(\sigma_jg_j)\le 2(k+\df)\,.
\end{array}
\end{equation}
Problem \eqref{eq:hier.puVa} can be solved numerically using semidefinite programming \cite{ben2001lectures}.
It is due to the fact that every SOS polynomial $\sigma$ of degree $2t$ has the form $\sigma=v_t^\top Gv_t$ for some positive semidefinite matrix $G$ (which is denoted by $G\succeq 0$, i.e., $G$ is symmetric and each eigenvalue of $G$ is nonnegative), with $v_t$ being the vector of all monomials $x^\alpha:=x_1^{\alpha_1}\dots x_n^{\alpha_n}$ of degree at most $t$.
Such a matrix $G$ is called a \emph{Gram matrix} associated to $\sigma$.
It is easy to see that for each $k \in \N$, $\rho_k^{(0)}$ is a lower bound of $f^\star$, that the sequence $(\rho_k^{(0)})_{k\in\N}$ is monotone nondecreasing, and converges to $f^\star$.

In the present paper, we answer the following two interesting  questions:
\begin{enumerate}
    \item How fast does $(\rho_k^{(0)})_{k\in\N}$ converge to $f^\star$? We show the convergence rate  $\mathcal{O}(k^{-1/c})$ for some constant $c>0$ depending on $f$ and $g_j$.
    \item Is there any explicit example to illustrate this rate of convergence? If $S$ is the unit ball, i.e., $m=1$ and $g_1=1-\|x\|_2^2$, the sequence $(\rho_k^{(0)})_{k\in\N}$ converges to $f^\star$ with the rate $\mathcal{O}(k^{-1/65})$.
\end{enumerate}
    
\subsection{Background}
Positivity certificates have been studied throughout  history  of  the  development  of  real  algebraic geometry.
Nowadays it become a powerful tool for polynomial optimization thanks to the so-called Moment-SOS hierarchy (also known as ``Lasserre’s hierarchy") \cite{lasserre2001global}.
The convergence rate of the Moment-SOS hierarchy to the optimal value of a POP inherently depends on the complexity of the representation of positive polynomials.
Roughly speaking, obtaining a lower degree bound on the SOS polynomials involved in the positivity certificate allows one to improve the convergence rate of the corresponding Moment-SOS hierarchy.
How to find such lower degree bound is an interesting question and goes hand in hand with the quest of improving the convergence analysis of the Moment-SOS hierarchy. 
Let us review some of the standard
results on degree bounds of positivity certificates and the corresponding convergence rates of the Moment-SOS hierarchy.

Reznick provides in \cite{reznick1995uniform} a Positivstellensatz for positive definite forms with an explicit degree bound.
Namely, if $p$ is a positive definite form, i.e., $p$ is homogeneous and  positive except at the origin, then for all $k \in \N$ satisfying
\begin{equation}\label{eq:bound.reznick}
    k\ge \frac{{2nd(2d - 1)}}{{4\log 2}}\Theta (p) - \frac{{n + 2d}}{2}\,,
\end{equation}
then ${\| x \|^{2k}_2}p$ is a homogeneous SOS polynomial of degree $2(k + d)$, where $2d=\deg(p)$.
Here for each $h\in\R[x]$, 
\begin{equation}
    \Theta(h):=\frac{\sup_{x \in  \mathbb S^{n-1}}  h(x) }{\inf_{x \in  \mathbb S^{n-1}}  h(x) }\,.
\end{equation}
This yields a linear convergence rate of  $\mathcal{O}(\varepsilon^{-1})$ for the minimization of a polynomial (see  \cite[Theorem 6]{mai2021positivity}).

Powers and Reznick \cite{powers2001new} improve the existing degree bound available for P\'olya's Positivstellensatz \cite{polya1928positive}.
Explicitly, if $p$ is a homogeneous polynomial of degree $d$ positive on the simplex 
\begin{equation}\label{eq:simplex.def}
\begin{array}{l}
    \Delta_n=\{x\in\R^n\,:\,x_j\ge 0\,,\,j\in[n]\,,\,\sum_{j\in[n]}x_j=1\}\,,
\end{array}
\end{equation}
then for all $k\in\N$ satisfying 
\begin{equation}\label{eq:bound.polya}
    k\ge \frac{d(d-1)\|p\|}{2\min_{x\in\Delta_n}p(x)} -d\,,
\end{equation}
$(\sum_{j\in[n]}x_j)^kp$ has positive coefficients.
Here for each $h=\sum_{\alpha}h_\alpha x^\alpha\in\R[x]$, we note $\| h \|: = \max_\alpha   \frac{| {{h_\alpha }} |}{c_\alpha}$ with $c_\alpha := \frac{|\alpha|!}{\alpha_1!\dots\alpha_n!}$ for each $\alpha \in \N^n$. 
This yields a linear convergence rate of  $\mathcal{O}(\varepsilon^{-1})$ for the minimization of a homogeneous polynomial on the simplex.

Applying the result of Powers and Reznick,  Schweighofer \cite{schweighofer2004complexity} obtains a degree bound  for Schm\"udgen's Positivstellensatz \cite{schmudgen1991thek} claiming that given a semialgebraic $S \subset (-1,1)^n$ defined as in \eqref{eq:semial.set.def} and a  polynomial $f$ positive on $S$, then there exists a real $c>0$ depending on $S$ such that for all $k\in\N$ satisfying
\begin{equation}\label{eq:bound.schmudgen}
    k\ge c\df^2\left(1+\left(\df^2n^{\df}\frac{\|f\|}{f^\star}\right)^c\right)\,,
\end{equation}
one has $f\in \mathcal{P}_k$, where $\mathcal{P}_k$ is the truncated preordering of order $k\in\N$ associated with $S$:
\begin{equation}
\label{eq-preordering}
\begin{array}{l}
\mathcal P_k := \{ \sum_{\alpha  \in {{\{ 0,1\} }^m}} {{\sigma _\alpha }g_1^{{\alpha _1}} \ldots g_m^{{\alpha _m}}}  \, :\,  {\sigma _\alpha } \in \Sigma [ x ]\,,\,\deg({\sigma _\alpha }g_1^{{\alpha _1}} \ldots g_m^{{\alpha _m}})\le k \}
\,.
\end{array}
\end{equation}
Consequently, the corresponding SOS hierarchy of lower bounds $ (\rho_k^{\text{pre}})_{k\in\N}$, with
\begin{equation}
\begin{array}{l}
 \rho_k^{\text{pre}}:=\sup_{\lambda\in\R}\{\lambda\,:\,f-\lambda \in \mathcal{P}_k\}\,,\quad k\in\N\,,
\end{array}
\end{equation}
converges to $f^\star$ with the rate $\mathcal{O}(\varepsilon^{-c})$.
Nevertheless, the representation of $f-\lambda$ in $\mathcal{P}_k$ involves $2^m$ SOS polynomials.

Relying on the degree bound for Schm\"udgen’s Positivstellensatz, Nie and Schweighofer analyze in \cite{nie2007complexity} the complexity of Putinar's Positivstellensatz.
Namely, if $S\subset (-1,1)^n$, $g_1=L-\|x\|_2^2$ for some $L>0$ and $f$ is positive on $S$, then there exists a real $c>0$ depending on $S$ such that for all $k\in\N$ satisfying
\begin{equation}
    k\ge c\exp\left(\left(\df^2n^{\df}\frac{\|f\|}{f^\star}\right)^c\right)\,,
\end{equation}
one has $f\in \mathcal{Q}_k$, where $\mathcal{Q}_k$ stands for the truncated quadratic module of order $k\in\N$ associated with  $S$:
\begin{equation}
\label{eq-quadratic}
\begin{array}{l}
\mathcal Q_k := \{ \sigma_0+\sum_{j\in[m]}\sigma_jg_j  \, :\,  {\sigma _j} \in \Sigma [ x ]\,,\,\deg(\sigma _0)\le k\,,\,\deg(\sigma _jg_j)\le k \}
\,.
\end{array}
\end{equation}
Accordingly, the corresponding SOS hierarchy of lower bounds $(\rho_k^{\text{mod}})_{k\in\N}$, with
\begin{equation}
\begin{array}{l}
\rho_k^{\text{mod}}:=\sup_{\lambda\in\R}\{\lambda\,:\,f-\lambda \in \mathcal{Q}_k\}\,,\quad k\in\N\,,
\end{array}
\end{equation}
converges to $f^\star$ with the rate  $\mathcal{O}(\exp(\varepsilon^{-c}))$.
Despite of this exponential convergence rate, the representation of $f-\lambda$ in $\mathcal{Q}_k$ involves only $m+1$ SOS polynomials which is in deep contrast with  the exponential number of SOS polynomials involved in the representation in $\mathcal{P}_k$.

However, the convergence rate of Schweighofer is still comparable to the later one of Nie-Schweighofer when the semialgebraic set $S$ is defined by $m_1\le 1$ inequality constraints $g_1(x)\ge 0$ and $m_2$ equality constraints $h_i(x)=0$. 
In this case, it turn out that
\begin{equation}
\begin{array}{rl}
      \rho_k^{\text{pre}}=\sup\limits_{\lambda,\sigma_j,\eta_i}& \lambda \\
     \text{s.t.}& \lambda\in\R\,,\,\sigma_j\in \Sigma[x]\,,\,\eta_i\in\R[x]\\
     &f-\lambda =\sigma_0+\sum_{j\in[m_1]}\sigma_jg_j +\sum_{i\in[m_2]}\eta_ih_i\,,\\
     &\deg(\sigma_0)\le k\,,\,\deg(\sigma_jg_j)\le k\,,\,\deg(\eta_ih_i)\le k\,,
\end{array}
\end{equation}
and the sequence $(\rho_k^{\text{pre}})_{k\in\N}$ converges to $f^\star$ with the rate $\mathcal{O}(\varepsilon^{-c})$.
Here $[0]=\emptyset$.
In particular, $c=\frac{1}{2}$ when $f$ is homogeneous, $\df\le n$, $m_1=0$ and $m_2=1$ with $h_1=1-\|x\|_2^2$ as shown by Fang and Fawzi \cite{fang2020sum}. 
Thus they obtain the quadratic convergence rate $\mathcal{O}(k^{-2})$ for the minimization of a homogeneous polynomial on the unit sphere.
This improves upon the earlier linear convergence rate $\mathcal{O}(k^{-1})$ by Doherty and Wehner in \cite{doherty2012convergence}.

Recently Slot and Laurent \cite{slot2020improved} have provided several convergence rates for Lasserre’s
measure-based \emph{upper bounds} for polynomial optimization, on specific convex domains and reference measures. 
This is in contrast with the present work, where we provide a new convergence rate for the Moment-SOS hierarchy of lower bounds.

Our contribution is concerned with the case of basic semialgebraic sets having nonempty interiors.  Basically one obtains a convergence rate similar in spirit and magnitude of Schweighofer's bound $\bar c\varepsilon^{-c}$, but still based on the quadratic module $\mathcal{Q}_k$ (not the preordering $\mathcal{P}_k$) thanks to the prescribed denominator $(1+\|x\|_2^2)^k$ involved in Putinar-Vasilescu’s Positivstellensatz.
\if{
\hoang{

As another evidence,  the interrupted relaxations with the help of the prescribed denominator can exploit the new sparsity in dense POPs on the nonnegative orthant with any largest matrix size and therefore provide better lower bounds within less computing time than the standard one.

They might expect the exact value due to the non-prescribed but how does one avoid the perturbation in recent SDP solver? A negative answers is obtained with the very accurate SDP solvers, e.g., interior-point method because of the pertubed log barrier.

The rounding and bisection  can be a solution but accordingly we need to solve the SDP many times. In fact, these method doesn't have theoretical guarantee for the exactness when the exact results are floating numbers.}
}\fi
\subsection{Contribution}
The construction of the hierarchy of semidefinite relaxations \eqref{eq:hier.puVa} is based on the Positivstellensatz stated in Corollary \ref{coro:compact}. 
More explicitly, 
if $S$ has nonempty interior such that $g_1=L-\|x\|_2^2$ for some $L>0$ and $f$ is of degree at most $2\df$ such that $f$ is nonnegative on $S$, then there exist positive constants $\bar c$ and $c$ depending on $f,g_j$ such that for all $\varepsilon>0$, for all $k\ge \bar c\varepsilon^{-c}$, 
    \begin{equation}\label{eq:Pu-va.ref}
    \begin{array}{l}
(1+ \|x\|_2^2)^{k}(f+\varepsilon)=\sigma_0+\sum_{j\in[m]}\sigma_jg_j\,,
    \end{array}
    \end{equation}
    for some $\sigma_j\in\Sigma[x]$ being such that $\deg(\sigma_0)\le 2(k+\df)$ and $\deg(\sigma_jg_j)\le 2(k+\df)$.

 In order to prove \eqref{eq:Pu-va.ref}, we provide a degree bound on the weighted SOS polynomials for the homogenized Putinar-Vasilescu's Positivstellensatz \cite{putinar1999positive}.
    This is stated in Theorem \ref{theo:complex.putinar.vasilescu} as follows:
    If $f,g_1,\dots,g_m$ are homogeneous polynomials of even degrees such that $S$ has nonempty interior and $f$ is nonnegative on $S$, then there exist positive constants $\bar c$ and $c$ depending on $f,g_j$ such that for all $\varepsilon>0$, for all $k\ge \bar c\varepsilon^{-c}$, 
    \begin{equation}\label{eq:PuVa.perturb}
    \begin{array}{l}
         \|x\|_2^{2k}(f+\varepsilon\|x\|_2^{2\df})=\sigma_0+\sum_{j\in[m]}\sigma_jg_j\,,
    \end{array}
    \end{equation}
    for some homogeneous SOS polynomials $\sigma_j$ being such that $\deg(\sigma_0)=\deg(\sigma_jg_j)=2(k+\df)$.
    Here a polynomial $p$ is homogeneous of degree $2 t$ if
$p(\lambda x)=\lambda^{2 t} p(x)$ for all $x\in\R^n$ and each $\lambda\in\R$.
Remark that the original version of Putinar-Vasilescu's Positivstellensatz in \cite{putinar1999positive} does not include any degree bound on the weighted SOS polynomials $\sigma_j$ involved in the representation \eqref{eq:PuVa.perturb}.
Our proof of Theorem \ref{theo:complex.putinar.vasilescu} consists of three main steps:
\begin{enumerate}
    \item Construct iteratively some positive ``weight''  functions $\psi_j$ such that $f+\varepsilon -\sum_{j\in[m]} \psi_jg_j$ is positive on $[-1,1]^n$.
    The idea of this step is similar in spirit to the proof of the inductive property in \cite[Proposition 3.1]{schulze2014schm} and relies on the Lojasiewicz inequality.
    \item Approximate $\sqrt{\psi_j}$ with the multivariate Bernstein polynomial $q_j$ on $[-1,1]^n$ such that the  polynomial $H=f+\varepsilon -\sum_{j\in[m]} q_j^2g_j$ is positive on the unit sphere $\mathbb{S}^{n-1}$.
    \item Apply Reznick's  Positivstellensatz \cite{reznick1995uniform} to the homogenization of $H$.
\end{enumerate}
The complexity analysis of every step is derived to get the final degree bound $\bar c\varepsilon^{-c}$.

Afterwards, we obtain in Corollary \ref{eq:dehomo} the same degree bound for the dehomogenized Putinar-Vasilescu's Positivstellensatz. 
This improves the bound $\mathcal{O}(\exp(\varepsilon^{-c}))$  obtained in our previous work \cite{mai2021positivity}, based on Nie-Schweighofer's complexity result \cite{nie2007complexity} for Putinar's Positivstellensatz \cite{putinar1993positive}.
Corollary \ref{eq:dehomo} yields the convergence rate $\mathcal{O}(\varepsilon^{-c})$ for the corresponding hierarchy of relaxations for polynomial optimization on general (not necessarily compact) basic semialgebraic sets.

\subsection{Technical insights}
We start to recall the two main steps in the proof of Nie and Schweighofer \cite{nie2007complexity} for the degree bound of SOS polynomials involved in Putinar's Positivstellensatz:
\begin{enumerate}
    \item Find a large enough $k\in\N$  such that the polynomial
    \begin{equation}\label{eq:update.poly}
    \begin{array}{l}
         F=f +\varepsilon - \lambda\sum_{j\in[m]} (g_j-1)^{2k} g_j
    \end{array}
    \end{equation}
     is positive on  $[-1,1]^n$.
    An estimate $k\ge O(\varepsilon^{-c})$ is obtained.
    Here $\varepsilon>0$ measures how close the polynomial $f$ (assumed to be nonnegative on $S$) is to have a zero on $S$.
    \item Apply Schm\"udgen's Positivstellensatz to $F$ on $[-1,1]^n$.
\end{enumerate}
Notice that Schweighofer's degree bound of Schm\"udgen's Positivstellensatz is exponential in the degree of the given positive polynomial ($n^{\df}$ in \eqref{eq:bound.schmudgen}).
Accordingly, Nie and Schweighofer obtain an exponential bound  $n^{\mathcal{O}(\varepsilon^{-c})}$ in the second step since $\deg(F)\sim Ck$ as $k\to\infty$ for some positive constant $C$.

One notable difference in our proof is that the weight $\lambda(g_j-1)^k$ in \eqref{eq:update.poly} is replaced by a non-differentiable positive function $\psi_j$. 
Surprisingly, we can prove that the square root $\sqrt{\psi_j}$ is a Lipschitz continuous function. 
Thus each $\sqrt{\psi_j}$ can be approximated with a Bernstein polynomials $q_i$ on $[-1,1]$. 
Here, the advantage of using Bernstein polynomials is that the  approximation error between $\sqrt{\psi_j}$ and $q_i$ decreases with a rate which only depends on a Lipschitz constant of $\sqrt{\psi_j}$, and $|q_i|$ is upper bounded by the supremum of $\sqrt{\psi_j}$ on $[-1,1]^n$.

Next, we apply Reznick's Positivstellensatz to the homogeneous polynomial $\tilde H$ obtained from the homogenization of
\begin{equation}
    \begin{array}{l}
         H:=f+\varepsilon -\sum_{j\in[m]}q_j^2g_j \,,
    \end{array}
\end{equation}
being such that the bounds of $\tilde H$  and $H$ on the unit sphere are the same.
The important point to note here is that the degree bound of Reznick's Positivstellensatz is quadratic in the degree of $\tilde H$ and linear in the ratio $\Theta(\tilde H)$ (see \eqref{eq:bound.reznick}).
This is in deep contrast with  Schm\"udgen's Positivstellensatz, as there is no exponential dependency in these two quantities. 
This leads to the difference between  our convergence rate $\mathcal{O}(\varepsilon^{-c})$ and Nie-Schweighofer's rate  $\mathcal{O}(\exp(\varepsilon^{-c}))$.

One may ask whether with the same techniques from our proof, one could improve the existing degree bound for Putinar's  Positivstellensatz.
We have tried to apply the degree bound \eqref{eq:bound.polya} of P\'olya's Positivstellensatz to $H$ after a change of coordinate, but unfortunately this leads to the same bound as Nie and Schweighofer.
The underlying reason is that the norm $\|p\|$ in \eqref{eq:bound.polya} depends on the coefficients of $p$.
In our situation, $p$ coincides with $H$ and the coefficients of $H$ are bounded by a value involving the coefficients of  the Bernstein polynomials.
The bound on the largest coefficient, even for a univariate Bernstein polynomial, seems to be exponential in the approximation order $t$, namely,
$\binom{2t}{t}\sim {\frac {4^{t}}{\sqrt {\pi t}}}$ as $t\to\infty$.
The same issue occurs when we apply the degree bound of Schm\"udgen's Positivstellensatz instead of the one of P\'olya's  Positivstellensatz.

\section{Notation and definitions}
In this section, we introduce mandatory notation and definitions.
With $x := (x_1,\dots,x_n)$, let $\R[x]$ stands for the ring of real polynomials and let 
$\Sigma[x]\subset\R[x]$ be the subset of sum-of-squares (SOS) of  polynomials.
Let us note $\R[x]_t$ and $\Sigma[x]_t$ the respective restrictions of these two sets to polynomials of degree at most $t$ and $2t$. 

Given $\alpha = (\alpha_1,\dots,\alpha_n) \in \N^n$, we note $|\alpha| := \alpha_1 + \dots + \alpha_n$.
Denote $\N^n_t=\{\alpha\in\N^n\,:\,|\alpha|\le t\}$ for each $t\in\N$.
Let $(x^\alpha)_{\alpha\in\N^n}$ 
be the canonical basis of monomials for  $\R[x]$ (ordered according to the graded lexicographic order) and 
$v_t(x)$ be the vector of all monomials up to degree $t$, with length ${\binom {n+t} n}$.
A polynomial $f\in\R[x]_t$ is written as  
$f(x)\,=\,\sum_{| \alpha | \leq t} f_\alpha\,x^\alpha\,=\,\bar{f}^\top v_t(x)$, 
where $\bar{f}:=(f_\alpha)_{\alpha}\in\R^{\binom {n+t} n}$ is its vector of coefficients in the canonical basis.
The
\textit{degree-$t$ homogenization} $\tilde f$ associated to $f\in\R[x]_t$ is a  homogeneous polynomial of degree $t$ in $n+1$ variables, defined by 
$\tilde f ( x,x_{n + 1}) := x_{n + 1}^{t} f ( x/x_{n + 1} )$.
A \textit{positive definite form} is a nonnegative homogeneous polynomial which is positive everywhere except at the origin. 

For a given real-valued sequence $y=(y_\alpha)_{\alpha\in\N^n}$, let us define the \emph{Riesz linear functional} $L_y:\R[ x ] \to \R$ by $f\mapsto {L_y}( f ) := \sum_{\alpha} f_\alpha y_\alpha$, for all $ f\in\R[x]$.
We say that a real infinite (resp. finite) sequence $( y_\alpha)_{\alpha  \in \N^n}$ (resp. $( y_\alpha)_{\alpha  \in \N^n_t}$) has a \emph{representing measure} if there exists a finite Borel measure $\mu$ such that $y_\alpha  = \int_{\R^n} {x^\alpha d\mu(x)}$ is satisfied for every $\alpha  \in {\N^n}$ (resp. $\alpha  \in {\N^n_t}$). In this case, $( y_\alpha)_{\alpha  \in \N^n}$ is called the \emph{moment sequence} of $\mu$. 
Next, given $y=(y_\alpha)_{\alpha  \in \N^n}$ and $d\in \N^*$,
the \emph{moment matrix} $M_d(y)$ of degree $d$ associated to $y$  is the real symmetric matrix of size $\binom{n+d}{d}$ defined by $M_d(y) := ( y_{\alpha  + \beta })_{\alpha,\beta\in \N^n_d} $. 
Let  $g = \sum_{\gamma} g_\gamma x^\gamma  \in \R[x]$. 
The localizing matrix $M_d(gy)$ of degree $d$ associated with $y$ and $g$ is the real symmetric matrix of the size $\binom{n+d}{d}$ given by $M_d(gy) := (\sum_\gamma  {{g_\gamma }{y_{\gamma  + \alpha  + \beta }}})_{\alpha, \beta\in \N^n_d}$.
\section{Representation theorems and degree bounds}

In this section, we derive representations of polynomials nonnegative on semialgebraic sets together with degree bounds.
We extend these representations to the set of continuous functions being nonnegative on compact domains.
\subsection{Polynomials nonnegative on general semialgebraic sets}
We analyze the complexity of Putinar-Vasilescu's Positivstellensatz \cite{putinar1999positive} in the following theorem:
\begin{theorem}\label{theo:complex.putinar.vasilescu}
(Homogenized representation)
Let $g_1,\dots,g_m$ be homogeneous polynomials of even degrees such that the semialgebraic set
\begin{equation}
    S:=\{x\in\R^n\,:\,g_1(x)\ge 0\dots,g_m(x)\ge 0\}
\end{equation} 
has nonempty interior.
Let $f$ be a homogeneous polynomial of degree $2\df$ for some $\df\in\N$ such that $f$ is nonnegative on $S$.
Then there exist positive constants $\bar c$ and $c$ depending on $f,g_1,\dots,g_m$ such that for all $\varepsilon>0$, for all $k\in\N$ satisfying
\begin{equation}
 k\ge \bar c\varepsilon^{-c}\,,
\end{equation}
there exist homogeneous SOS polynomials $\sigma_0,\dots,\sigma_m$ such that 
\begin{equation}\label{eq:degree.SOS}
    \deg(\sigma_0)=\deg(\sigma_1g_1)=\dots=\deg(\sigma_mg_m)=2(k+\df)
\end{equation}
and 
\begin{equation}\label{eq:represent}
    \|x\|_2^{2k}(f+\varepsilon\|x\|_2^{2\df})=\sigma_0+\sigma_1g_1+\dots+\sigma_mg_m\,.
\end{equation}
Moreover, if $m=1$ and $g_1=x_n^2-\|x'\|_2^2$ with  $x':=(x_1,\dots,x_{n-1})$, then $c=65$.
\end{theorem}
The proof of Theorem \ref{theo:complex.putinar.vasilescu} is postponed to Appendix  \ref{proof:complex.putinar.vasilesc}.
\begin{remark}
The proof of Theorem \ref{theo:complex.putinar.vasilescu} provides additional information that each SOS polynomial $\sigma_i$ involved in \eqref{eq:represent} can be chosen as the (single) square of a homogeneous polynomial, for $i\in[m]$.
\end{remark}
The following corollary is a direct consequence of Theorem \ref{theo:complex.putinar.vasilescu}.
\begin{corollary}
\label{eq:dehomo}
(Dehomogenized representation)
Let $g_1,\dots,g_m$ be polynomials such that  the semialgebraic set
\begin{equation}
    S:=\{x\in\R^n\,:\,g_1(x)\ge 0\dots,g_m(x)\ge 0\}
\end{equation} 
has nonempty interior.
Let $f$ be a polynomial nonnegative on $S$.
Denote $\df: = \lfloor \deg(f)/2\rfloor +1$.
Then there exist positive constants $\bar c$ and $c$ depending on $f,g_1,\dots,g_m$ such that for all $\varepsilon>0$, for all $k\in\N$ satisfying
\begin{equation}
 k\ge \bar c\varepsilon^{-c}\,,
\end{equation}
there exist SOS polynomials $\sigma_0,\dots,\sigma_m$ such that 
\begin{equation}
    \deg(\sigma_0)\le 2(k+\df)\quad\text{and}\quad\deg(\sigma_jg_j)\le 2(k+\df)\,,\,j=1,\dots,m\,,
\end{equation}
and 
\begin{equation}
   \theta^{k}(f+\varepsilon\theta^{\df})=\sigma_0+\sigma_1g_1+\dots+\sigma_mg_m\,,
\end{equation}
where $\theta:=1+\|x\|_2^2$.
Moreover, if $m=1$ and $g_1=1-\|x\|_2^2$, then $c=65$.
\end{corollary}
\begin{proof}
The proof of Corollary \ref{eq:dehomo} is similar to the proof of \cite[Theorems 4 and 5]{mai2021positivity}.
We recall the basic ingredients. 
Let $\tilde S$ be a homogenized version of $S$, defined by 
\begin{equation}
    \tilde S:=\{(x,x_{n+1})\in\R^{n+1}\,:\,\tilde g_j(x,x_{n+1})\ge 0\,,\,j\in[m]\}\,,
\end{equation}
 with $\tilde g_j ( x,x_{n + 1}) := x_{n + 1}^{2d_{g_j}} g_j ( x/x_{n + 1} )$ being the degree-$2d_{g_j}$ homogenization of $g_j$ and $d_{g_j}:=\lceil \deg(g_j)/2\rceil$, for $j\in[m]$.
Then the proof consists of three steps:
\begin{enumerate}
\item Prove that the degree-$2d_f$ homogenization of $f$, denoted by ${\tilde f}$, is nonnegative on $\tilde S$.
\item  Use Theorem \ref{theo:complex.putinar.vasilescu} to obtain a representation of $\tilde f$ together with the degree bound on SOS polynomials.
\item  Obtain a representation of $f$ by evaluating the representation of ${\tilde f}$ at $x_{n+1} = 1$.
\end{enumerate}
To apply Theorem \ref{theo:complex.putinar.vasilescu}, we need to show that if $S$ has nonempty interior, then $\tilde S$ has nonempty interior.
This statement holds since when $a$ belongs to the interior of $S$, one has $\tilde g_j ( a,1)=g_j(a)>0$, implying that  $(a,1)$ belongs to the interior of $\tilde S$.
\end{proof}
Note that the ice cream constraint $x_{n+1}^2-\|x\|_2^2$ is the degree-2 homogenization associated to the ball constraint $1-\|x\|_2^2$.
\subsection{Polynomials nonnegative on compact semialgebraic sets}

The following corollary is deduced from Corollary \ref{eq:dehomo}.
\begin{corollary}\label{coro:compact}
Let $g_1,\dots,g_m$ be polynomials such that $g_1:=L-\|x\|_2^2$ for some $L>0$ and the semialgebraic set
\begin{equation}
    S:=\{x\in\R^n\,:\,g_1(x)\ge 0\dots,g_m(x)\ge 0\}
\end{equation} 
has nonempty interior.
Let $f$ be a polynomial nonnegative on $S$.
Denote $\df: = \lfloor \deg(f)/2\rfloor +1$.
Then there exist positive constants $\bar c$ and $c$ depending on $f,g_1,\dots,g_m$ such that for all $\varepsilon>0$, for all $k\in\N$ satisfying
\begin{equation}
 k\ge \bar c\varepsilon^{-c}\,,
\end{equation}
there exist SOS polynomials $\sigma_0,\dots,\sigma_m$ such that 
\begin{equation}
    \deg(\sigma_0)\le 2(k+\df)\quad\text{and}\quad\deg(\sigma_jg_j)\le 2(k+\df)\,,\,j=1,\dots,m\,,
\end{equation}
and 
\begin{equation}
   (1+\|x\|_2^2)^{k}(f+\varepsilon)=\sigma_0+\sigma_1g_1+\dots+\sigma_mg_m\,.
\end{equation}
Moreover, if $m=1$ and $L=1$, then $c=65$.
\end{corollary}
\begin{proof}
By using Theorem \ref{eq:dehomo}, 
there exist positive constants $\bar c$ and $c$ depending on $f,g_1,\dots,g_m$ such that for all $\varepsilon>0$, for all $k\in\N$ satisfying
$
 k\ge \bar c\varepsilon^{-c}
$,
there exist SOS polynomials $\sigma_0,\dots,\sigma_m$ such that 
\begin{equation}
    \deg(\sigma_0)\le 2(k+\df)\quad\text{and}\quad\deg(\sigma_jg_j)\le 2(k+\df)\,,\,j=1,\dots,m\,.
\end{equation}
and 
\begin{equation}
   \theta^{k}(f+\varepsilon\theta^{\df})=\sigma_0+\sigma_1g_1+\dots+\sigma_mg_m\,,
\end{equation}
where $\theta:=1+\|x\|_2^2$.
In addition,
\begin{equation}
    (L+1)^{\df}-\theta^{\df}=(L-\|x\|_2^2)\sum_{j=0}^{\df-1}(L+1)^{\df-1-j}\theta^{j}=s_1g_1\,,
\end{equation}
where $s_1=\sum_{j=0}^{\df-1}(L+1)^{\df-1-j}\theta^{j}$ is an SOS polynomial of degree $2(\df-1)$.
From this,
\begin{equation}
\begin{array}{rl}
    \theta^{k}[f+\varepsilon (L+1)^{\df}] &   =\theta^{k}(f+\varepsilon\theta^{\df})+\varepsilon\theta^{k}[(L+1)^{\df}-\theta^{\df}]\\
     & =\sigma_0+(\varepsilon s_1\theta^k+\sigma_1)g_1+\sum_{j=2}^m\sigma_jg_j\,,
\end{array}
\end{equation}
which yields the desired result.
\end{proof}
\begin{remark}
We can apply the technique used in the proof of Corollary \ref{coro:compact}, which consists of replacing the perturbation $\varepsilon\theta^{\df}$ by $\varepsilon$, to represent polynomials nonnegative on $\R^n$. 
Let us consider an arbitrary large positive constant $L$ and a polynomial $f$ of degree $2\df$ which is nonnegative on $\R^n$. 
Then, thanks to \cite[Theorem 3.2]{mai2021positivity}, for any $\varepsilon>0$, for all $k\in\N$ such that $k\ge \mathcal{O}(\varepsilon^{-1})$, $\theta^k(f+\varepsilon\theta^{\df})$ is an SOS polynomial, so that $\theta^k(f+\varepsilon)=\sigma_0+\sigma_1(L-\|x\|_2^2)$ for some SOS polynomials $\sigma_i$, $i=0,1$.
This is the so-called ``big ball trick''.
This representation yields a linear convergence rate $\mathcal{O}(\varepsilon^{-1})$ for the minimization of  polynomials on $\R^n$.
\end{remark}

\section{Polynomial optimization}
This section is concerned with some applications to polynomial optimization.

Consider the following POP:
\begin{equation}\label{eq:constrained.problem.poly}
\begin{array}{l}
f^\star:=\inf\limits_{x \in S} f( x)\,,
\end{array}
\end{equation}
where $f\in\R[x]$ and
\begin{equation}\label{eq:semial.set.def.2}
    S=\{x\in\R^n\,:\,g_j(x)\ge 0\,,\,j\in[m]\}\,,
\end{equation}
for some $g_j\in\R[x]$, $j\in[m]$.
 Assume that $S$ has nonempty interior and $f^\star>-\infty$. 

Recall that $\theta=1+\|x\|_2^2$. Let $\df := \lfloor \deg(f)/2\rfloor +1$ and $d_{g_j}: = \lceil {\deg ( {{g_j}} )/2} \rceil $, $j\in\{0\}\cup[m]$.

\subsection{General case}
In this subsection, we improve the convergence rate of the Moment-SOS hierarchy described in \cite[Theorem 4.3]{mai2021positivity},  based on Putinar-Vasilescu's Positivstellensatz \cite{putinar1999positive}.

Let $\varepsilon>0$ be fixed.
Consider the hierarchy of semidefinite programs indexed by $k\in\N$: 
\begin{equation}\label{eq:dual-sdp}
\begin{array}{rl}
\tau_k^{(\varepsilon)}: = \inf &{L_y}( {{\theta ^k}( {f +  \varepsilon {\theta ^{\df}}} )} )\\
\text{s.t.}&y = {(y_\alpha )_{\alpha  \in \N^n_{2( {\df + k})}}} \subset \R\,,\\
&{M_{k + \df}}( {y})\succeq 0\,,\\
&{M_{k + \df - {d_{g_j}}}}( {{g_j}y}) \succeq 0 ,\;j \in[m]\,,\\
&{L_y}( {{\theta ^k}}) = 1\,.
\end{array}
\end{equation}
\begin{theorem}\label{theo:constr.theo}
For every $k\in\N$, the dual of \eqref{eq:dual-sdp} reads as:
\begin{equation}\label{eq:primal.problem}
\begin{array}{l}
   \rho _k^{ (\varepsilon)}:= \sup_{\lambda  \in \R} \{ {\lambda  :\ {\theta ^k}\,( {f - \lambda +  \varepsilon \,{\theta ^{\df}} } ) \in \mathcal Q_{k + \df}} \}\,.
\end{array}
\end{equation}
The following statements hold:
\begin{enumerate}
\item For all $k\in\N$,
\begin{equation}
    \rho_k^{(\varepsilon)}\le\rho_{k+1}^{(\varepsilon)}\le f^\star\,.
\end{equation}
\item Assume that problem \eqref{eq:constrained.problem.poly} has an optimal solution $x^\star$. 
Then there exists positive constants $\bar c$ and $c$ depending on $f,g_1,\dots,g_m$ such that for all $k\ge \bar c\varepsilon^{-c}$, 
\begin{equation}
    0\le \rho_k^{(\varepsilon)}-f^\star \le \varepsilon \theta {( {{x^\star}})^{\df}}\,.
\end{equation}
\item Strong duality holds for all orders $k$ of the  primal-dual problems \eqref{eq:dual-sdp}-\eqref{eq:primal.problem}.
\end{enumerate}
\end{theorem}
The proof of Theorem~\ref{theo:constr.theo} is exactly the same as the proof of \cite[Theorem 7]{mai2021positivity}.
The second statement relies on Corollary \ref{eq:dehomo}.
The third statement is due to the Slater condition \cite[Theorem 3.1]{vandenberghe1996semidefinite} since $S$ has nonempty interior.

\subsection{Compact case}
In this subsection, we consider the case when $S$ is compact by assuming that a ball constraint is present. 
We can then remove the perturbation term $\varepsilon\theta^{\df}$ in the hierarchy based on Putinar-Vasilescu's Positivstellensatz, described in the previous subsection.

Assume that $g_1=L-\|x\|_2^2$ for some $L>0$.
Consider the hierarchy of semidefinite programs indexed by $k\in\N$: 
\begin{equation}\label{eq:dual-sdp.0}
\begin{array}{rl}
{\tau_k^{(0)}}: = \inf &{L_y}( {\theta ^k}f  )\\
\text{s.t.}&y = {(y_\alpha )_{\alpha  \in \N^n_{2( {\df + k})}}} \subset \R\,,\\
&{M_{k + \df}}( {y})\succeq 0\,,\\
&{M_{k + \df - d_{g_j}}}( {{g_j}y}) \succeq 0 ,\;j \in[m]\,,\\
&{L_y}( {{\theta ^k}}) = 1\,.
\end{array}
\end{equation}
\begin{theorem}\label{theo:constr.theo.0}
For every $k\in\N$, the dual of \eqref{eq:dual-sdp.0} reads as:
\begin{equation}\label{eq:primal.problem.0}
\begin{array}{l}
   {\rho _k^{(0)}}:= \sup_{\lambda\in\R} \{ {\lambda\, :\, {\theta ^k}\,( {f - \lambda } ) \in \mathcal Q_{k + \df}} \}\,.
\end{array}
\end{equation}
The following statements hold:
\begin{enumerate}
\item For all $k\in\N$,
\begin{equation}
    \rho_k^{(0)}\le\rho_{k+1}^{(0)}\le f^\star\,.
\end{equation}
\item There exist positive constants $\bar c$ and $c$ depending on $f,g_1,\dots,g_m$ such that 
\begin{equation}
    0\le f^\star-\rho_k^{(0)}\le \left(\frac{\bar c}{k}\right)^{\frac{1}{c}}
\end{equation}
\item Strong duality holds for all orders $k$ of the  primal-dual problems \eqref{eq:dual-sdp.0}-\eqref{eq:primal.problem.0}.
\end{enumerate}
\end{theorem}

\begin{proof}
The first and third statements of  Theorem~\ref{theo:constr.theo.0}  can be proved similarly to the ones of Theorem~\ref{theo:constr.theo}.
Let us proof the second statement.
By using Corollary \ref{coro:compact}, there exist positive constants $\bar c$ and $c$ depending on $f,g_1,\dots,g_m$ such that for any $\varepsilon>0$, for all $k\ge \bar c\varepsilon^{-c}$, 
\begin{equation}
    \theta^{k}(f-f^\star+\varepsilon)\in \mathcal Q_{k+\df}\,.
\end{equation}
Let $K\in\N$. Set $\epsilon=(\frac{\bar c}{K})^{\frac{1}{c}}$. Then $\epsilon>0$ and $K=\bar c\epsilon^{-c}$, so that
\begin{equation}
    \theta^{K}(f-f^\star+\epsilon)\in \mathcal Q_{K+\df}\,.
\end{equation}
It implies that $f^\star-\epsilon$ is a feasible solution of \eqref{eq:primal.problem.0} with relaxation order $K$, yielding
\begin{equation}
    0\le f^\star-\rho_K^{(0)}\le f^\star - (f^\star-\epsilon) = \epsilon= \left(\frac{\bar c}{K}\right)^{\frac{1}{c}}\,.
\end{equation}
Hence the desired result follows.
\end{proof}

\section{Conclusion}
We have provided a new degree bound on the  sum-of-squares (SOS) polynomials involved in Putinar-Vasilescu’s Positivstellensatz.
The Moment-SOS hierarchy of semidefinite relaxations based on this Positivstellensatz provide a sequence of lower bounds on the minimum of a polynomial on a basic compact semialgebraic set.
Moreover, this sequence of lower bounds converges to the minimum  with the rate $\mathcal{O}(\varepsilon^{-c})$, at prescribed  accuracy $\varepsilon>0$. 
We emphasize that this improves upon the previous convergence rate $\mathcal{O}(\exp(\varepsilon^{-c}))$ stated by Nie and Schweighofer \cite{nie2007complexity} for Putinar's Positivstellensatz.

A topic of further investigation is the analysis of the convergence rate of the Moment-SOS hierarchy for lower bounds in some special cases of basic (compact) semialgebraic sets. 
A fairly straightforward idea is to find the explicit constant $\alpha$ in the Łojasiewicz inequality stated in Lemma \ref{lem:lojas.ineq}.
We could then proceed analogously to the proof of the rate $\mathcal{O}(\varepsilon^{-65})$ for the minimization of a polynomial on the unit ball.
\if{
\hoang{
\paragraph{Discussion.}
Recently people from theoretical computer science expect the exact value in polynomial optimization.
They like to use non-prescribed denominators in the certificates with SOS of rational functions and therefore have to run several times the SDPs with non-variables (but guessed) lower bound and use bisection to adjust the lower bound.

Clearly Reznick and Putinar-Vasilescu provide the positivity certificates for polynomials nonnegative on noncompact semialgebraic sets with prescribed denominators. 
It leads to the perturbation in the values returned by the hierarchy of semidefinite relaxations based on these certificates. 

Of course, it should be better to have positivity certificates with non-prescribed denominators and then obtain the exact value (maybe as expected) by solving the SDP relaxations for polynomial optimization. 
However it cannot be denied that many state of the art SDP solvers, e.g, interior-point method and first-order methods,  still provide the perturbed value. 
Nevertheless, bisection method indeed does not give us the exact values in practice.
}
}\fi
\paragraph{Acknowledgements.}
The first author was supported by the MESRI funding from EDMITT.
The second author was supported by the Tremplin ERC Stg Grant ANR-18-ERC2-0004-01 (T-COPS project) and by the FMJH Program PGMO (EPICS project) and  EDF, Thales, Orange et Criteo.
This work has benefited from the European Union's Horizon 2020 research and innovation programme under the Marie Sklodowska-Curie Actions, grant agreement 813211 (POEMA) as well as from the AI Interdisciplinary Institute ANITI funding, through the French ``Investing for the Future PIA3'' program under the Grant agreement n$^{\circ}$ANR-19-PI3A-0004.

\appendix
\section{Appendix}

\subsection{Preliminary material}
This subsection presents some important lemmas that we use to prove the main results.

Given $\Omega\subset \R^n$, the distance of $a\in\R^n$ to $\Omega$ is denoted by $\dist(a,\Omega)$.
Denote by $B(a,r)$ (resp. $B^\circ (a,r)$) the closed (resp. open) ball centered at $a\in\R^n$ with radius $r>0$.

\begin{lemma}\label{lem:lojas.ineq}(Łojasiewicz inequality \cite[Corollary 2.6.7]{bochnak2013real})
Let $r>0$ and the semialgebraic set $S:=\{x\in\R^n\,:\,g_j(x)\ge 0\,,\,j\in[m]\}$,
where $g_1,\dots,g_m$ are polynomials.
Then there exist positive constants $\alpha$ and $C$ such that, for all $x$ in $ B(0,r)$,
\begin{equation}
    \dist(x,S)^{\alpha}\le -C \min\{g_1(x),\dots,g_m(x),0\}\,.
\end{equation}
\end{lemma}

Given an open set $U\subset\R^n$ and a differentiable function $\varphi:U\to \R$, denote by $\nabla \varphi(x)=[\partial_{x_1} \varphi(x),\dots,\partial_{x_n} \varphi(x)]$ the gradient of $\varphi$ at $x\in U$. 
Given $x = (x_1,\dots,x_n) \in \R^n$, let $x':=(x_1,\dots,x_{n-1})$.
\begin{lemma}\label{lem:lojas.ineq.ball}(Łojasiewicz inequality with ice cream constraint)
Let $g:=x_n^2-\|x'\|_2^2$ and $Z:=\{x\in\R^n\,:\,g(x)= 0\}$.
Then for all $x\in\R^n$,
\begin{equation}\label{eq:Loja.ball}
    \dist(x,Z)^{2}\le \frac{|g(x)|}{2} \,.
\end{equation}
\end{lemma}
\begin{proof}
If $x\in Z$, both sides of \eqref{eq:Loja.ball} are zeros.
Let $x\in\R^n\backslash Z$ be fixed.
Then $d(x,Z)^2=\min_y\{\|x-y\|_2^2\,:\,g(y)=0\}$.
Assume that $(y,\mu)\in\R^n\times \R$ satisfies the Karush–Kuhn–Tucker conditions:
\begin{equation}\label{eq:kkt}
    \begin{cases}
    \nabla_y \|x-y\|_2^2 =\mu \nabla_y g(y)\,,\\
    g(y)=0\,.\\
    \end{cases}
\end{equation}
The first condition of \eqref{eq:kkt} implies that $2(x-y)=\mu\begin{bmatrix}
-2y'\\
2y_n
\end{bmatrix}$, so $x'-y'=-\mu y'$ and $x_n-y_n=\mu y_n$.
Assume that $\mu\notin \{ 1,-1\}$.
Then $y'=\frac{x'}{1-\mu}$ and $y_n=\frac{x_n}{1+\mu}$.
Since $g(y)=y_n^2-\|y'\|_2^2=0$, $y_n=\pm \|y'\|_2$.

Let us consider the first case $y_n=\|y'\|_2$.
Then $\frac{x_n}{1+\mu}=\frac{\|x'\|_2}{1-\mu}$.
It implies that $\mu=\frac{x_n-\|x'\|_2}{x_n+\|x'\|_2}$.
Note that $x_n\ne -\|x'\|_2$ since $g(x)\ne 0$.
From this, $y'=\frac{(x_n+\|x'\|_2)x'}{2\|x'\|_2}$ and $y_n=\frac{x_n+\|x'\|_2}{2}$.
Thus, $\|x-y\|_2^2=\frac{(x_n-\|x'\|_2)^2}{2}$.

Similarly, if we consider the case $y_n=-\|y'\|_2$, then $\|x-y\|_2^2=\frac{(x_n+\|x'\|_2)^2}{2}$.

Let us consider the case of $\mu\in\{1,-1\}$. Assume that $\mu=1$. Then $x'=0$ and $y_n=\frac{x_n}{2}$.
From this and the fact that $0=g(y)=y_n^2-\|y'\|_2^2$, we obtain  $\|y'\|_2^2=\frac{x_n^2}{4}$. 
It implies that $\|x-y\|_2^2=\|y'\|_2^2+(x_n-y_n)^2=\frac{x_n^2}{4}+\frac{x_n^2}{4}=\frac{x_n^2}{2}=\frac{(x_n-\|x'\|_2)^2}{2}$.
Thus, $\|x-y\|_2^2=\frac{(x_n-\|x'\|_2)^2}{2}$.

Similarly, if we consider the case $\mu=-1$, then $\|x-y\|_2^2=\frac{(x_n+\|x'\|_2)^2}{2}$.

Thus, 
\begin{equation}
\begin{array}{rl}
   d(x,Z)^2  & \le \frac{1}{2}\min\{(x_n-\|x'\|_2)^2,(x_n+\|x'\|_2)^2\}  \\
     & \le \frac{1}{2}\sqrt{(x_n-\|x'\|_2)^2(x_n+\|x'\|_2)^2}\\
     &=\frac{1}{2}|x_n^2-\|x'\|_2^2|=\frac{|g(x)|}{2}\,,
\end{array}
\end{equation}
yielding \eqref{eq:Loja.ball}.
\end{proof}

A real-valued function $f : U \to \R$ for some $U\subset \R^n$ is called $L$-Lipschitz (or Lipschitz) continuous on $K\subset U$ if there exits a real $L>0$ such that
$|f(x)-f(y)| \le L\|x-y\|_2$, for all $x,y\in K$.
In this case, $L$ is called the Lipschitz constant of $f$ on $K$.
Given an open set $U\subset \R^n$, a function $f : U \to \R$  is called locally Lipschitz continuous on $K\subset U$ if for every $x\in K$ there exists a neighborhood $W\subset U$ of $x$ such that  $f$ is Lipschitz continuous on ${W\cap K}$.

The following lemma is similar in spirit to \cite[Section 2.4, Lemma 2]{perko2013differential}:
\begin{lemma}
\label{lem:local.Lipschitz}
Given an open set $U\subset \R^n$, if the function $f:U \to \R$ is locally Lipschitz on a compact set $K\subset U$, then $f$ is Lipschitz on $K$.
\end{lemma}
\begin{proof}
Since $f$ is locally Lipschitz on $K$, for each $x\in K$ there is some $r_x>0$ and $L_x>0$ such that $B(x,r_x)\subset U$ and $f$ is $L_x$-Lipschitz on $B(x,r_x)\cap K$.
Then the sets $B(x,\frac{1}2r_x)$, $x\in K$ form an open cover of $K$.
Due to the compactness of $K$, there exists a finite subsequence of $B(x,\frac{1}2r_x)$, $x\in K$ covering $K$. 
For convenience, denote these by $B(x_k,\frac{1}2r_k)$ and $L_{k}:=L_{x_k}$, $k\in[l]$.
Let $M:=\sup_{x\in K}|f(x)|$,  $r:=\frac{1}2 \min_{k\in[l]} r_k$, $L_0:=\frac{2M}{r}$ and $L:=\max\{L_0,L_k\,:\,k\in[l]\}$. Then $L$ is a Lipschitz constant of $f$ on $K$.
To see this, pick $x,y\in K$. 
If $\|x-y\|_2\ge r$ then we see that $\frac{|f(x)-f(y)|}{\|x-y\|_2}\le \frac{2M}{r}=L_0\le L$. 
If $\|x-y\|_2<r$, then for some $x_k$ we have $x\in B(x_k,\frac{1}2r_k)$. 
Then $y\in B(x_k,r_k)$ and so $|f(x)-f(y)|\le L_k\|x-y\|_2\le L\|x-y\|_2$.
\end{proof}

\begin{lemma}\label{lem:Kirszbraun}
(Kirszbraun's theorem \cite{kirszbraun1934zusammenziehende})
If $U$ is a subset of $\R^n$ and $f : U \to \R$
is a Lipschitz continuous function, then there is a Lipschitz continuous function
$F: \R^n \to \R$
that extends $f$ and has the same Lipschitz constant as $f$.
Moreover the extension is provided by 
\begin{equation}
\begin{array}{l}
F(x):=\inf _{u\in U}\{f(u)+L_f \|x-u\|_2\}\,,
\end{array}
\end{equation}
where $L_f$ is the Lipschitz constant of $f$ on $U$.
\end{lemma}

We recall basic properties of the multivariate Bernstein polynomials described, e.g., in \cite{hildebrandt1933linear,heitzinger2002simulation}.
\begin{definition}
(Multivariate Bernstein polynomials)   Let $d\in\mathbb{N}^n$ and $f \in C([0, 1]^n)$.
The polynomials
\begin{equation}
\begin{array}{l}
B_{f,d}(x):=
\sum_{k_1=0}^{d_1} \dots \sum_{k_n=0}^{d_n}  f\left(\frac{k_1}{d_1},\dots,\frac{k_n}{d_n}\right)
\prod_{j=1}^n \left[\binom{d_j}{k_j} x_j^{k_j} (1-x_j)^{d_j-k_j} \right]
\end{array}
\end{equation}
are called the multivariate Bernstein polynomials of $f$.
\end{definition} 
Note that $ \deg(B_{f,d})=\sum_{j\in[n]}d_j$ and the binomial identity implies
\begin{equation}\label{eq:bound.sup.Bernstein}
    \begin{array}{l}
        \sup_{x\in[0,1]^n} |B_{f,d}(x)|\le \sup_{x\in[0,1]^n} |f(x)|\,.
    \end{array}
\end{equation}
\if{
\hoang{
To prove \eqref{eq:bound.sup.Bernstein}, we apply the multinomial identity to get
\begin{equation}
   1= [x_j+(1-x_j)]^{d_j}=\sum _{k=0}^{d_j}\binom{d_j}{ k}x_j^{d_j-k}(1-x_j)^{k}\,.
\end{equation}
}
}\fi
\if{
\begin{lemma}\label{lem:A.7}
For all $ a,b\in\mathbb{R}^n$ and all $ c\in\mathbb{R}_+^n$ the inequality
\begin{equation}
    \displaystyle \left( \sum a_k b_k c_k \right)^2 \le
\left( \sum a_k^2 c_k \right) \left( \sum b_k^2 c_k \right)
\end{equation}
holds.
\end{lemma}

\begin{lemma}\label{lem:B.4}   For all $ x\in\mathbb{R}$
$$\sum_{k=0}^{n} (k-nx)^2 \binom{n}{k} x^k (1-x)^{n-k} = nx(1-x).
$$
For all $ x\in[0,1]$ we have $ x(1-x)\le 1/4$ and hence
$$\displaystyle 0 \le \sum_{k=0}^{n} (k-nx)^2 \binom{n}{k} x^k (1-x)^{n-k} \le \frac{n}{4}.
$$
\end{lemma}
}\fi
\begin{lemma}\label{lem:error.Lipschitz}
(Error bound \cite[Theorem 7.12]{heitzinger2002simulation})   If $ f \in C([0,1]^n)$ is $L$-Lipschitz, namely
$|f(x)-f(y)| \le  L \Vert x-y\Vert _2
$
on $[0,1]^n$, then for all $d\in\N^n$, the inequality
\begin{equation}
\begin{array}{l}
|B_{f,d}(x) - f(x) | \le
\frac{L}{2} ( \sum_{j=1}^n \frac{1}{d_j} )^{\frac{1}2}
\end{array}
\end{equation}
holds for all $x\in[0,1]^n$.
\end{lemma}

Let $e:=(1,\dots,1)\in\R^n$.
As a consequence of Lemma \ref{lem:error.Lipschitz}, we obtain the following result after a change of coordinates.

\begin{lemma}\label{lem:app.poly}
If $ f \in C([0,1]^n)$ is $L$-Lipschitz, namely 
$|f(x)-f(y)| \le  L \Vert x-y\Vert _2
$
on $[-1,1]^n$, then for all $k\in\N^{\ge 1}$, the inequality
\begin{equation}\label{eq:rate.approx}
    \left|B_{y\mapsto f(2y-e),ke}\left(\frac{x+e}{2}\right) - f(x) \right| \le
{L} \biggl(\frac{n}{k} \biggr)^{\frac{1}2}
\end{equation}
holds for all $x\in [-1,1]^n$.
Moreover, we have
\begin{equation}\label{eq:bound.sup.bernstein}
    \begin{array}{l}
        \sup_{x\in[-1,1]^n} |B_{y\mapsto f(2y-e),ke}\left(\frac{x+e}{2}\right)|\le \sup_{x\in[-1,1]^n} |f(x)|\,.
    \end{array}
\end{equation}

\end{lemma}
\begin{proof}
Define $g:[0,1]^n\to \R$ by $g(x):=f(2x-e)$.
Let us compute a Lipschitz constant of $g$.
With $x,y\in[0,1]^n$, by the Lipschitz continuity of $f$, we have
\begin{equation}
\begin{array}{rl}
   |g(x)-g(y)|  =&|f(2x-e)-f(2y-e)| \\
   \le  & L\|2x-e-2y+e\|_2\\
   =&2L\|x-y\|_2\,.
\end{array}
\end{equation}
Then $2L$ is a Lipschitz constant of $g$.
Let $k\in\N^{\ge 1}$.
Using Lemma \ref{lem:error.Lipschitz}, we get that for all $x\in[0,1]^n$,
\begin{equation}
\begin{array}{l}
|B_{g,ke}(x) - g(x) | \le\frac{2L}{2} ( \sum_{j=1}^n \frac{1}{k} )^{\frac{1}2}= {L} ( \frac{n}{k} )^{\frac{1}2}\,.
\end{array}
\end{equation}
Let $y\in[-1,1]^n$. Then $\frac{y+e}2\in[0,1]$ implies that
\begin{equation}
\begin{array}{l}
|B_{g,ke}(\frac{y+e}2) - f(y) | =|B_{g,ke}(\frac{y+e}2) - g(\frac{y+e}{2}) | \le {L} ( \frac{n}{k} )^{\frac{1}2}\,.
\end{array}
\end{equation}
yielding \eqref{eq:rate.approx}.

In addition, from \eqref{eq:bound.sup.Bernstein},
\begin{equation}
    \begin{array}{rl}
        \sup_{y\in[-1,1]}|B_{g,ke}(\frac{y+e}{2})|&=\sup_{x\in[0,1]^n} |B_{g,ke}(x)|\\
        & \le \sup_{x\in[0,1]^n} |g(x)|\\
        &=\sup_{y\in[-1,1]}|g(\frac{y+e}{2})|=\sup_{y\in[-1,1]}|f(y)|\,,
    \end{array}
\end{equation}
which yields \eqref{eq:bound.sup.bernstein}.
\end{proof}

For each $h\in\R[x]$, let
\begin{equation}
    \Theta(h):=\frac{\sup_{x \in  \mathbb S^{n-1}}  h(x) }{\inf_{x \in  \mathbb S^{n-1}}  h(x) }\,.
\end{equation}
%
For later use recall the following theorem.
\begin{lemma}(Reznick \cite[Theorem 3.12]{reznick1995uniform})\label{lem:homogeneous.sos}
Suppose that $p \in \R[ x]$ is a positive definite form of degree $2d$, for some $d\in\N$. 
Then for all $k \in \N$ satisfying
\begin{equation}
    k\ge \frac{{2nd(2d - 1)}}{{4\log 2}}\Theta (p) - \frac{{n + 2d}}{2}\,,
\end{equation}
${\| x \|^{2k}_2}p$ is a homogeneous SOS polynomial of degree $2(k + d)$.
\end{lemma}

\subsection{The proof of Theorem \ref{theo:complex.putinar.vasilescu}}
\label{proof:complex.putinar.vasilesc}
Recall that $[l]:=\{1,\dots,l\}$ for $l\in\N^{\ge 1}$.
Given real value functions $p,q$, we use the notation $\{p * q\}=\{x\in\R^n\,:\, p(x) * q(x)\}$, where $* \in \{  =, \ge, \le, >, < \}$.
Given a real value function $p$ on $\Omega\subset\R^n$, note $\|p\|_{\Omega}:=\sup_{x\in \Omega}|p(x)|$.
With $\Omega\subset \R^n$, denote by $\text{int}(\Omega)$ the interior of $\Omega$.

Given $U,V \subseteq \R^n$ and $r \in \R$, note $U+V=\{u+v\,:\,u\in U\,,\,v\in V\}$ and $rU=\{ru\,:\,u\in U\}$.
Given a function $f:U\to\R$ and $A\subset U\subset \R^n$ such that $A=-A$, $f$ is called even on $A$ if $f(-x)=f(x)$ for all $x\in A$.
Denote by $\mathbb{S}^{n-1}$  the unit sphere of $\R^n$. 

To begin the proof, let us fix $\varepsilon>0$.
By assumption, $\deg(f)=2\df$, $\deg(g_j)=2d_{g_j}$ for some $d,d_{g_j}\in\N$,  for $j\in [m]$.
\subsubsection{Construction of the positive weight functions}
\label{sec:construct}
For $j\in[m]$, define 
\begin{equation}
    S_j:=\{x\in\R^n\,:\,g_i(x)\ge 0\,,\,i\in[j]\}\,.
\end{equation}
Obviously, we have $S_m=S$. Note $S_0:=\R^n$ and $f_m:=f$.

We will prove that there exist functions  $\bar\varphi_{m}:\R^n\to\R$ such that the following conditions hold:
\begin{enumerate}
    \item $\bar\varphi_{m}$ is positive, even and bounded from above by  $C_{\bar\varphi_{m}}=\bar r_m\varepsilon^{-r_m}$  on $ B(0,\sqrt{n}+m)$ for some positive constants $\bar r_m$ and $r_m$ independent of $\varepsilon$.
    \item $\bar\varphi_{m}$ is Lipschitz with Lipschitz constant $L_{\bar\varphi_{m}}=\bar t_j\varepsilon^{-t_m}$ for some positive constants $\bar t_m$ and $t_m$ independent of $\varepsilon$.
    \item 
       $f_{m-1}:=f_{m}+\frac{\varepsilon}{2}-\bar\varphi_{m}^2g_{m}$ satisfies:
       \begin{enumerate}
           \item $f_{m-1}\ge 0$ on $S_{m-1}\cap  B(0,\sqrt{n}+m-1)$;
           \item  $f_{m-1}\le C_{f_{m-1}}$ on $ B(0,\sqrt{n}+{m})$, where $C_{f_{m-1}}=\bar c_{m-1}\varepsilon^{-c_{m-1}}$  for some positive constants $\bar c_{m-1}$ and $c_{m-1}$ independent of $\varepsilon$;
           \item $f_{m-1}$ is Lipschitz  on $ B(0,\sqrt{n}+{m})$ with Lipschitz constant $L_{f_{m-1}}=\bar l_{m-1}\varepsilon^{-l_{m-1}}$ for some positive constants $\bar l_{m-1}$ and $l_{m-1}$ independent of $\varepsilon$.
       \end{enumerate}
\end{enumerate}

Let
\begin{equation}\label{eq:constant.M}
    M_m:=\inf_{x\in S_{m}\cap  B(0,\sqrt{n}+m)} \frac{f(x)+\frac{\varepsilon}{2}}{g_m(x)}\,.
\end{equation}
\paragraph{The constant $M_m$ is a positive real number.} 
Let $C_{g_m}=\|g_m\|_{ B(0,\sqrt{n}+m)}$.
We claim that $\frac{\varepsilon}{2 C_{g_m}}<M_m<\infty$. 
Indeed, if $z$ is a feasible solution of \eqref{eq:constant.M}, $z\in S$ yielding $f(z)\ge 0$ so that
\begin{equation}
    \frac{f(z)+\frac{\varepsilon}{2}}{g_m(z)} \ge \frac{\varepsilon}{2g_m(z)}\ge \frac{\varepsilon}{2C_{g_m}}\,.
\end{equation}
From this, we have $M_m>\frac{\varepsilon}{2C_{g_m}}$.
On the other hand,
there exists $a\in\R^n$ such that $g_j(a)>0$ for $j\in[m]$ since $S$ has nonempty interior. 
For $j\in[m]$, since $g_j$ is homogeneous, $a=0$ yields   $g_j(a)=0$.
It implies that $a\ne 0$.
With $\bar a=\frac{a}{\|a\|_2}\in B(0,1) \subset B(0,\sqrt{n}+m)$, we obtain $g_j(\bar a)>0$ for $j\in[m]$ since 
\begin{equation}
    g_j(\bar a)=g_j\left(\frac{a}{\|a\|_2}\right)=\frac{g_j( a)}{\|a\|_2^{2d_{ g_j}}}>0\,,\,\forall j\in[m]\,.
\end{equation}
Thus, $\bar a$ is a feasible solution of \eqref{eq:constant.M} which yields 
\begin{equation}\label{eq:bound.M}
    \frac{\varepsilon}{2C_{g_m}}\le M_m\le \frac{f(\bar a)+\frac{\varepsilon}{2}}{g_m(\bar a)}\le\frac{C_f+\frac{\varepsilon}{2}}{g_m(\bar a)}< \infty\,,
\end{equation}
where $C_f:=\|f\|_{ B(0,\sqrt{n}+m)}$.

Let $\psi_m: \R^n\to \R$ be the function defined by 
\begin{equation}
    \psi_m(x):=\begin{cases}
    \max\{M_m,\frac{f(x)+\frac{\varepsilon}{2}}{g_m(x)}\}&\text{if }g_m(x)<0\,,\\
    M_m&\text{otherwise}.
    \end{cases}
\end{equation}
\paragraph{The function $f+\frac{\varepsilon}{2} -\psi_mg_m$ is nonnegative on $S_{m-1}\cap  B(0,\sqrt{n}+m)$.} Namely, we claim that
\begin{equation}\label{eq:firstterm}
    f+\frac{\varepsilon}{2}-\psi_mg_m\ge 0\text{ on }S_{m-1}\cap  B(0,\sqrt{n}+m)\,.
\end{equation}
Let $y\in S_{m-1}\cap  B(0,\sqrt{n}+m)$.
If $g_m(y)< 0$, then
\begin{equation}
\begin{array}{rl}
    f(y)+\frac{\varepsilon}{2}-\psi_m(y)g_m(y)= & f(y)+\frac{\varepsilon}{2}-g_m(y) \max\{M_m,\frac{f(y)+\frac{\varepsilon}{2}}{g_m(y)}\} \\
     \ge  & f(y)+\frac{\varepsilon}{2}-g_m(y) \frac{f(y)+\frac{\varepsilon}{2}}{g_m(y)}=  0\,.
\end{array}
\end{equation}
Otherwise, $g_m(y)\ge 0$ gives
\begin{equation}
\begin{array}{rl}
    f(y)+\frac{\varepsilon}{2}-\psi_m(y)g_m(y)&=  f(y)+\frac{\varepsilon}{2}-g_m(y) M_m \\
       & \begin{cases} \ge f(y)+\frac{\varepsilon}{2}-g_m(y) \frac{f(y)+\frac{\varepsilon}{2}}{g_m(y)}= 0 & \text{if } g_m(y)>0\,, \\
     =f(y)+\frac{\varepsilon}{2}\ge  0 & \text{if } g_m(y)=0\,,
     \end{cases}
\end{array}
\end{equation}
since $y\in S$ is a feasible solution of \eqref{eq:constant.M}.
\paragraph{The function $\psi_m$ is positive, even on $ B(0,\sqrt{n}+m)$ and continuous on $S_{m-1}\cap  B(0,\sqrt{n}+m)$.}
It is easy to see that $\psi_m$ is bounded from below by $ M_m$ and continuous on $  B(0,\sqrt{n}+m)\backslash \{ g_m=0\}$ since the max function $(t_1,t_2)\mapsto \max\{t_1,t_2\}$ is continuous. 

We claim that $\psi_m$ is continuous on $S_{m-1}\cap  B(0,\sqrt{n}+m)\cap \{ g_m =0\}$. 
Indeed, let us consider a sequence $(y_l)_l \subset S_{m-1}\cap  B(0,\sqrt{n}+m)\cap \{g_m <0\}$ such that $y_l \to \bar y\in S_{m-1}\cap  B(0,\sqrt{n}+m)\cap \{ g_m =0\}$. 
Then $g_m(y_l)\to 0^-$ and $f(y_l)\to f(\bar y)\ge 0$ (since $\bar y\in S$) yielding that $\frac{f(y_l)+\frac{\varepsilon}{2}}{g_m(y_l)}\to -\infty$.
It implies that $\max\{M_m,\frac{f(y_l)+\frac{\varepsilon}{2}}{g_m(y_l)}\}\to M_m$.
Thus, $\psi_m=M_m$ on a sufficiently small neighborhood of any point in $S_{m-1}\cap  B(0,\sqrt{n}+m)\cap \{  g_m =0\}$.
On the other hand, $\psi_m$ is even, i.e., $\psi_m(x)=\psi_m(-x)$ due to the fact that $f,g_1,\dots,g_m$ are even and $ B(0,\sqrt{n}+m)=- B(0,\sqrt{n}+m)$.

\paragraph{The upper bound of $\psi_m$ depends on $\varepsilon$.}
It follows from  \eqref{eq:bound.M} that $\psi_m = M_m$ on $ B(0,\sqrt{n}+m) \cap \{g_m \geq 0\}$ and so is bounded from above by $\frac{f(\bar a)+\frac{\varepsilon}{2}}{g_m(\bar a)}$ .

Let us compute an upper bound of $\psi_m$  on $S_{m-1}\cap B(0,\sqrt{n}+m) \cap \{g_m< 0\}$.
Let $y\in S_{m-1}\cap B(0,\sqrt{n}+m)$ be such that $g_m(y)<0$ and $\frac{f(y)+\frac{\varepsilon}{2}}{g_m(y)}> M_m$.
Then $\psi_m(y)=\frac{f(y)+\frac{\varepsilon}{2}}{g_m(y)}$.
By using the Łojasiewicz inequality (see Lemma \ref{lem:lojas.ineq}), there exist $C_m>0$ and $\alpha_m>0$ depending on $g_1,\dots,g_m$ such that for all $ x\in S_{m-1}\cap B(0,\sqrt{n}+m)\cap \{g_m< 0\}$,
\begin{equation}\label{eq:lojasiewicz.ineq}
    \dist(x,S)^{\alpha_m}\le -C_m \min\{g_1(x),\dots,g_m(x),0\}= -C_mg_m(x)\,.
\end{equation}
Let $\delta_m=\frac{1}{C_m}(\frac{\varepsilon}{2L_f})^{\alpha_m}$, where $L_f$ is a Lipschitz constant of $f$ on $ B(0,\sqrt{n}+m)$.
Consider the following two cases:
\begin{itemize}
    \item Case 1: $g_m(y)\le -\delta_m< 0$. Then 
    \begin{equation}
        \psi_m(y)=  \frac{f(y)+\frac{\varepsilon}{2}}{g_m(y)}=\frac{-f(y)-\frac{\varepsilon}{2}}{-g_m(y)}\le  \frac{C_f}{-g_m(y)} \le  \frac{C_f}{\delta_m}\le C_mC_f\left(\frac{2L_f}{\varepsilon}\right)^{\alpha_m}  \,.
    \end{equation}
    \item Case 2: $-\delta_m\le g_m(y)< 0$.
    Let $z\in S$ such that $\dist(y,S)=\|y-z\|_2$.
    Then  \eqref{eq:lojasiewicz.ineq} turns to $-f(y)\le \frac{\varepsilon}{2}$ according to
    \begin{equation}
    \begin{array}{rl}
        -f(y) & \le -f(z)+ L_f\|y-z\|_2 \le L_f \dist(y,S) \\
         & \le L_f(-C_mg_m(y))^{\frac{1}{\alpha_m}}\le L_f(C_m\delta_m)^{\frac{1}{\alpha_m}}=\frac{\varepsilon}{2}\,.
    \end{array}
    \end{equation}
    From this, we obtain
        \begin{equation}
        M_m<  \frac{f(y)+\frac{\varepsilon}{2} }{g_m(y)}=\frac{-f(y)-\frac{\varepsilon}{2} }{-g_m(y)} \le\frac{\frac{\varepsilon}{2}-\frac{\varepsilon}{2}}{-g_m(y)} =0< M_m\,.
    \end{equation}
    The contradiction indicates that this case does not occur.
\end{itemize}
Thus, the bound is given as follows
\begin{equation}\label{eq:bound.psi}
    \sup_{x\in S_{m-1}\cap  B(0,\sqrt{n}+m)}\psi_m(x)\le \max\left\{ \frac{f(\bar a)+\frac{\varepsilon}{2}}{g_m(\bar a)}, C_mC_f\left(\frac{2L_f}{\varepsilon}\right)^{\alpha_m}  \right\}=:C_{\psi_m}\,.
\end{equation}
Moreover, we obtain the inclusion
\begin{equation}\label{eq:inclusion.domain}
    S_{m-1}\cap B(0,\sqrt{n}+m) \cap \{\xi_m \ge M_m\}\cap \{g_m\le 0\} \subset \{g_m\le -\delta_m\}\,,
\end{equation}
where $\xi_m(x)= \frac{f(x)+\frac{\varepsilon}{2}}{g_m(x)}$.
Let $\varphi_m$ be the square root of $\psi_m$, i.e., $\varphi_m(x):=\sqrt{\psi_m(x)}$.
Then $\varphi_m$ is well-defined on $ B(0,\sqrt{n}+m)$ since $\psi_m$ is positive.
Moreover, $\varphi_m$ is finitely bounded from above on $S_{m-1}\cap  B(0,\sqrt{n}+m)$ by $C_{\varphi_m}:=\sqrt{C_{\psi_m}}$ and $\varphi_m$ is continuous on $S_{m-1}\cap  B(0,\sqrt{n}+m)$ since $\xi_m$ is continuous on $S_{m-1}\cap  B(0,\sqrt{n}+m)$.

\paragraph{The function $\varphi_m$ is Lipschitz continuous on $S_{m-1}\cap B(0,\sqrt{n}+m-1)$.}
Keep in mind that $\psi_m$ is defined by the constant function $M_m$ and the function $\xi_m$.
Since $\varphi_m$ takes the constant value $\sqrt{M_m}$ on $ B(0,\sqrt{n}+m) \backslash(\{\xi_m \ge M_m\}\cap \{g_m\le 0\}) $, $\varphi_m$ is Lipschitz continuous on $ B(0,\sqrt{n}+m) \backslash(\{\varphi_m \ge M_m\}\cap \{g_m\le 0\}) $ with zero Lipschitz constant.

On the other hand, $\varphi_m= \sqrt{\xi_m}$ on $B(0,\sqrt{n}+m) \cap \{\xi_m \ge M_m\}\cap \{g_m\le 0\}$.
As a consequence of \eqref{eq:inclusion.domain}, we have
\begin{equation}\label{eq:2seteqal}
\begin{array}{rl}
     & S_{m-1}\cap B(0,\sqrt{n}+m) \cap \{\xi_m \ge M_m\}\cap \{g_m\le 0\} \\
   =  & S_{m-1}\cap B(0,\sqrt{n}+m) \cap \{\xi_m \ge M_m\}\cap \{g_m\le -\delta_m\}\,.
\end{array}
\end{equation}

It implies that 
\begin{equation}\label{eq:def.varphi}
    \varphi_m(x)=\begin{cases}
    \sqrt{\xi_m(x)}&\text{if }x\in S_{m-1}\cap B(0,\sqrt{n}+m) \cap \{\xi_m \ge M_m\}\cap \{g_m\le -\delta_m\}\,,\\
    \sqrt{M_m}&\text{if }x\in (S_{m-1}\cap B(0,\sqrt{n}+m)) \backslash( \{\xi_m \ge M_m\}\cap \{g_m\le -\delta_m\})\,.
    \end{cases}
\end{equation}

The second equality is due to the fact that $\varphi_m=\sqrt{M_m}$ on $(S_{m-1}\cap B(0,\sqrt{n}+m)) \backslash( \{\xi_m \ge M_m\}\cap \{g_m\le 0\})$ and 
\begin{equation}
\begin{array}{rl}
    &(S_{m-1}\cap B(0,\sqrt{n}+m)) \backslash( \{\xi_m \ge M_m\}\cap \{g_m\le 0\})   \\
    = & (S_{m-1}\cap B(0,\sqrt{n}+m))\backslash[S_{m-1}\cap B(0,\sqrt{n}+m) \cap \{\xi_m \ge M_m\}\cap \{g_m\le 0\}]\\
    = & (S_{m-1}\cap B(0,\sqrt{n}+m))\backslash[S_{m-1}\cap B(0,\sqrt{n}+m) \cap \{\xi_m \ge M_m\}\cap \{g_m\le -\delta_m\}]\\
    = &(S_{m-1}\cap B(0,\sqrt{n}+m))\backslash(\{\xi_m \ge M_m\}\cap \{g_m\le -\delta_m\})\,.
\end{array}
\end{equation}

Set 
\begin{equation}
    w_m:=\min\left\{1,\frac{\delta_m}{2L_{g_m}},\frac{\varepsilon{\delta_{m}^2}}{8C_{g_m}[L_{f}C_{g_{m}}+(C_f+\frac{\varepsilon}{2})L_{g_{m}}]}\right\}\,.
\end{equation}
and
\begin{equation}\label{eq:sum.ball}
    W_m:=\left( B(0,\sqrt{n}+m-1) \cap \{\xi_m \ge M_m\}\cap \{g_m\le -\delta_m\}\right)+w_mB(0,1)\,.
\end{equation}
Then $ B(0,\sqrt{n}+m-1) \cap \{\xi_m \ge M_m\}\cap \{g_m\le -\delta_m\}\subset W_m$. 
Next, we prove that
\begin{equation}\label{eq:open.inclusion}
    W_m\subset B(0,\sqrt{n}+m) \cap \{\xi_m \ge \frac{M_m}2\}\cap \{g_m\le -\frac{\delta_m}{2}\}\,.
\end{equation}
Let $y\in W_m$. Then $y=z+w_m u$ for some $z\in S_{m-1}\cap B(0,\sqrt{n}+m-1) \cap \{\xi_m \ge M_m\}\cap \{g_m\le -\delta_m\}$ and for some $u\in B(0,1)$.
Combining $\|z\|_2\le \sqrt{n}+m-1$, $0<w_m< 1$ and $\|u\|_2\le 1$, one has $\|y\|_2\le \|z\|_2+w_m\|u\|_2\le \sqrt{n}+m$, yielding $y\in B(0,\sqrt{n}+m)$.
Since $g_m(z)\le -\delta_m$, we have
\begin{equation}
    g_m(y) \le g_m(z)+L_{g_m}\|y-z\|_2 \le -\delta_m+L_{g_m}w_m\|u\|_2\le -\delta_m+L_{g_m}
    \frac{\delta_m}{2L_{g_m}}\le  -\frac{\delta_m}{2}\,,
\end{equation}
where $L_{g_m}$ is a Lipschitz constant of $g_m$ on $B(0,\sqrt{n}+m)$.
Thus $y\in \{g_m\le -\frac{\delta_m}{2}\}$. 
This in turn implies
\begin{equation}
\begin{array}{rl}
    &{|\xi_{m}(y)-\xi_{m}(z)|}\\
    =& {\left|\frac{f(y)+\frac{\varepsilon}{2}}{g_{m}(y)}-\frac{f(z)+\frac{\varepsilon}{2}}{g_{m}(z)}\right|}\\
    = &\frac{|(f(y)+\frac{\varepsilon}{2})g_{m}(z)-(f(z)+\frac{\varepsilon}{2})g_{m}(y)|}{|g_{m}(y)||g_{m}(z)|}\\
    \le&\frac{2}{\delta_{m}^2}{|(f(y)+\frac{\varepsilon}{2} - f(z)-\frac{\varepsilon}{2})g_{m}(z)+(f(z)+\frac{\varepsilon}{2})(g_{m}(z)-g_{m}(y))|}\\
    \le& \frac{2}{\delta_{m}^2}{[|f(y) - f(z)||g_{m}(z)|+(|f(z)|+\frac{\varepsilon}{2})|g_{m}(z)-g_{m}(y)|]}\\
    \le &\frac{2}{\delta_{m}^2}{[L_{f}\|y - z\|_2C_{g_{m}}+(C_f+\frac{\varepsilon}{2})L_{g_{m}}\|z-y\|_2]}\\
    \le &\frac{2}{\delta_{m}^2}{[L_{f}C_{g_{m}}+(C_f+\frac{\varepsilon}{2})L_{g_{m}}]w_m\|u\|_2}\le \frac{\varepsilon}{4C_{g_m}}\le\frac{M_m}2\,.
\end{array}
\end{equation}
Since $\xi_m(z)\ge M_m$, we obtain $\xi_m(y)\ge \xi_m(z)-|\xi_m(y)-\xi_m(z)|\ge M_m-\frac{M_m}{2}=\frac{M_m}{2}$, yielding $y\in\{\xi_m\ge \frac{M_m}{2}\}$, which concludes the proof of \eqref{eq:open.inclusion} and ensures that $ \sqrt{\xi_m}$ is well-defined on $W_m$.

Let us prove that $\sqrt{\xi_m}$ is  Lipschitz on $W_m$. 
Let $y,z\in W_m$ such that $y\ne z$. Then
\begin{equation}
\begin{array}{rl}
    &\frac{|\sqrt{\xi_{m}(y)}-\sqrt{\xi_{m}(z)}|}{\|y - z\|_2}\\
    =&\frac{|\xi_{m}(y)-\xi_{m}(z)|}{\|y - z\|_2(\sqrt{\xi_{m}(y)}+\sqrt{\xi_{m}(z)})}\\
    \le& \frac{\left|\frac{f(y)+\frac{\varepsilon}{2}}{g_{m}(y)}-\frac{f(z)+\frac{\varepsilon}{2}}{g_{m}(z)}\right|}{2\sqrt{\frac{M_{m}}2}\|y - z\|_2}\\
    \le &\frac{|(f(y)+\frac{\varepsilon}{2})g_{m}(z)-(f_{m}(z)+\frac{\varepsilon}{2})g_{m}(y)|}{2g_{m}(y)g_{m}(z)\sqrt{\frac{\varepsilon}{4C_{g_m}}}\|y - z\|_2}\\
    \le& \frac{2|(f(y)+\frac{\varepsilon}{2})g_{m}(z)-(f_{m}(z)+\frac{\varepsilon}{2})g_{m}(y)|}{\delta_{m}^2\sqrt{\frac{\varepsilon}{4C_{g_m}}}\|y - z\|_2}\\
    =& \frac{2|(f(y)+\frac{\varepsilon}{2} - f(z)-\frac{\varepsilon}{2})g_{m}(z)+(f(z)+\frac{\varepsilon}{2})(g_{m}(z)-g_{m}(y))|}{\delta_{m}^2\sqrt{\frac{\varepsilon}{4C_{g_m}}}\|y - z\|_2}\\
    \le& \frac{2[|f(y) - f(z)||g_{m}(z)|+(|f(z)|+\frac{\varepsilon}{2})|g_{m}(z)-g_{m}(y)|]}{\delta_{m}^2\sqrt{\frac{\varepsilon}{4C_{g_m}}}\|y - z\|_2}\\
    \le& \frac{2[L_{f}\|y - z\|_2C_{g_{m}}+(C_f+\frac{\varepsilon}{2})L_{g_{m}}\|z-y\|_2]}{\delta_{m}^2\sqrt{\frac{\varepsilon}{4C_{g_m}}}\|y - z\|_2}\\
    \le& \frac{2[L_{f}C_{g_{m}}+(C_{f}+\frac{\varepsilon}{2})L_{g_{m}}]}{\delta_{m}^2\sqrt{\frac{\varepsilon}{4C_{g_m}}}}=:L_{\sqrt{\xi_m}}
    \,,
\end{array}
\end{equation}
Thus, $L_{\sqrt{\xi_m}}$ is a Lipschitz constant of $\sqrt{\xi_m}$ on $W_m$.

Set $K:=S_{m-1}\cap B(0,\sqrt{n}+m-1)$, $K_1:=K \cap \{\xi_m \ge M_m\}\cap \{g_m\le -\delta_m\}$ and $K_2:=K \backslash (\{\xi_m \ge M_m\}\cap \{g_m\le -\delta_m\})$.
Note that $K=K_1\cup K_2$ and $K_1\cap K_2=\emptyset$.
From \eqref{eq:def.varphi}, $\varphi_m=\sqrt{\xi_m}$ on $K_1$ and $\varphi_m=\sqrt{M_m}$ on $K_2$.

To conclude that $\varphi_m$ is Lipschitz on  $K$ according to Lemma \ref{lem:local.Lipschitz} (see Figure \ref{fig:Lipschitz}), it is sufficient to prove that $\varphi_m$ is locally Lipschitz on $K$.

Explicitly, we will show that for all $z\in K$, $\varphi_m$ is Lipschitz on $B(z,\frac{w_m}2)\cap K$ with Lipschitz constant $L_{\sqrt{\xi_m}}$.
Let $z\in K$. Let $u,v\in B(z,\frac{w_m}2)\cap K$ and consider the following cases:
\begin{itemize}
    \item Case 1: $u,v\in K_1$. Then $u,v\in W_m$ by definition of $W_m$. Moreover, $\varphi_m(u)=\sqrt{\xi_m(u)}$ and $\varphi_m(v)=\sqrt{\xi_m(v)}$. In this case, by the Lipschitz continuity of $\sqrt{\xi_m}$ on $W_m$,
    \begin{equation}
        |\varphi_m(u)-\varphi_m(v)|=|\sqrt{\xi_m(u)}-\sqrt{\xi_m(v)}|\le L_{\sqrt{\xi_m}}\|u-v\|_2\,.
    \end{equation}
    
    \item Case 2: $u,v\in K_2$.
    In this case, $\varphi_m(u)=\varphi_m(v)=\sqrt{M_m}$, so that
    \begin{equation}
        |\varphi_m(u)-\varphi_m(v)|=0\le L_{\sqrt{\xi_m}}\|u-v\|_2\,.
    \end{equation}
    \item Case 3: $u\in K_1$ and $v\in K_2$. 
    We claim that $B(z,\frac{w_m}2) \subset W_m$. 
    Let $q\in B(z,\frac{w_m}2)$.
    Then
    $\|q-u\|_2\le\|q-z\|_2+\|z-u\|_2\le w_m$ yielding $q\in u+w_mB(0,1)\subset K_1+w_mB(0,1)\subset W_m$.
    Then $u,v\in B(z,\frac{w_m}2) \subset W_m$.
    Moreover, 
     $\varphi_m(u)=\sqrt{\xi_m(u)}$ and $\varphi_m(v)=\sqrt{M_m}$. 
According to the continuity of $\xi_m$ on $B(z,\frac{w_m}2)\subset W_m$ and the convexity of $B(z,\frac{w_m}2)$, there exists $y\in B(z,\frac{w_m}2)\cap \{\xi_m=M_m\}\cap \{tu+(1-t)v\,:\,t\in[0,1]\}$. 
Then with $y=\lambda u+(1-\lambda)v$ for some $\lambda \in[0,1]$, we have
\begin{equation}
\begin{array}{rl}
    |\varphi_m(u)-\varphi_m(v)|  \le&  |\varphi_m(u)-\varphi_m(y)| +|\varphi_m(y)-\varphi_m(v)|\\
      \le &|\sqrt{\xi_m(u)}-\sqrt{\xi_m(y)}|+ |\sqrt{M_m}-\sqrt{M_m}|\\
     \le &L_{\sqrt{\xi_m}}\|u-y\|_2\\
     \le &L_{\sqrt{\xi_m}}\|u-\lambda u-(1-\lambda)v\|_2\\
     \le &L_{\sqrt{\xi_m}}(1-\lambda)\|u-v\|_2\le L_{\sqrt{\xi_m}}\|u-v\|_2\,.
\end{array}
\end{equation}
\end{itemize}


\definecolor{uuuuuu}{rgb}{0.26666666666666666,0.26666666666666666,0.26666666666666666}
\definecolor{ududff}{rgb}{0.30196078431372547,0.30196078431372547,1}
\definecolor{xdxdff}{rgb}{0.49019607843137253,0.49019607843137253,1}

\begin{figure}
\begin{center}
\begin{tikzpicture}[line cap=round,line join=round,>=triangle 45,x=1cm,y=1cm]
\clip(-5.625503484414552,-1.4739416348527954) rectangle (4.670427553230122,6.286739626080435);
\draw [line width=1pt] (-4,0)-- (-4,4);
\draw [line width=1pt] (-4,4)-- (3,4);
\draw [line width=1pt] (3,4)-- (3,0);
\draw [line width=1pt] (-4,0)-- (3,0);
\draw [line width=1pt] (-5,4)-- (-5,0);
\draw [line width=1pt] (-4,5)-- (3,5);
\draw [line width=1pt] (4,4)-- (4,0);
\draw [line width=1pt] (-4,-1)-- (3,-1);
\draw [shift={(-4,4)},line width=1pt]  plot[domain=1.5707963267948966:3.141592653589793,variable=\t]({1*1*cos(\t r)+0*1*sin(\t r)},{0*1*cos(\t r)+1*1*sin(\t r)});
\draw [shift={(-4,0)},line width=1pt]  plot[domain=3.141592653589793:4.71238898038469,variable=\t]({1*1*cos(\t r)+0*1*sin(\t r)},{0*1*cos(\t r)+1*1*sin(\t r)});
\draw [shift={(3,0)},line width=1pt]  plot[domain=-1.5707963267948966:0,variable=\t]({1*1*cos(\t r)+0*1*sin(\t r)},{0*1*cos(\t r)+1*1*sin(\t r)});
\draw [line width=1pt] (-2,2) circle (1cm);
\draw [shift={(3,4)},line width=1pt]  plot[domain=0:1.5707963267948966,variable=\t]({1*1*cos(\t r)+0*1*sin(\t r)},{0*1*cos(\t r)+1*1*sin(\t r)});
\draw [line width=0.5pt,dash pattern=on 1pt off 1pt] (-2,2)-- (-3,2);
\draw (-2.8240509621781595,2.582458007885445) node[anchor=north west] {$\frac{w_m}2$};
\draw (-2.7550343103973985,5.8050452064458355) node[anchor=north west] {$\xi_m=M_m$};
\draw (0.068418167476441,5.934621601537185) node[anchor=north west] {$g_m=-\frac{\delta_m}2$};
\draw (2.082531391492065,5.848706322518915) node[anchor=north west] {$g_m=0$};
\draw (-3.9353369666069447,0.6669359543701648) node[anchor=north west] {$K$};
\draw (-4.597318852748257,-0.17814730453363528) node[anchor=north west] {$U$};
\draw (-3.7240661518809937,3.835998175259415) node[anchor=north west] {$\xi_m>M_m$};
\draw [line width=1pt] (0.9379764930716581,5.385317483249716)-- (0.9802306560168466,-2.2063471259027563);
\draw [line width=1pt] (2.8478555443725536,5.399402204231446)-- (2.8476024986336014,-1.952822148231619);
\draw [line width=1pt] (-2,5.3712327622679865)-- (-2,-2);
\draw (-1.935306587201276,2.340203844940255) node[anchor=north west] {$z$};
\draw [line width=1pt] (-2.4846107054887465,1.6951205860364544)-- (-1.5550191206945627,1.6951205860364544);
\draw (-2.725457915306049,1.67967795853755) node[anchor=north west] {$u$};
\draw (-1.7536121675666748,1.67967795853755) node[anchor=north west] {$v$};
\draw (-2.3592551697810674,1.674137237817215) node[anchor=north west] {$y$};
\draw (-0.8930372345532516,2.761338521264265) node[anchor=north west] {$W_m$};
\draw [line width=1pt] (0,5.371232762267986)-- (0,-2.093669358048915);
\draw (-1.245155259096503,3.510642639551735) node[anchor=north west] {$w_m$};
\draw [line width=0.5pt,dash pattern=on 1pt off 1pt] (-2,3.469795429734435)-- (0,3.469795429734435);
\begin{scriptsize}
\draw [fill=xdxdff] (-2,2) circle (1pt);
\draw [fill=ududff] (0.9802306560168466,-2.2063471259027563) circle (1pt);
\draw[color=ududff] (1.092908423870687,-1.9035256247955605) node {$W$};
\draw [fill=ududff] (2.276024986336014,-1.952822148231619) circle (1pt);
\draw[color=ududff] (2.4309569171350427,-1.6077464841792304) node {$A_1$};
\draw [fill=xdxdff] (-2,-2) circle (1pt);
\draw[color=xdxdff] (-1.8507982613108958,-1.6500006471244204) node {$C_1$};
\draw [fill=ududff] (-2.4846107054887465,1.6951205860364544) circle (1pt);
\draw [fill=ududff] (-1.5550191206945627,1.6951205860364544) circle (1pt);
\draw [fill=uuuuuu] (-2,1.695120586036454) circle (1pt);
\draw [fill=uuuuuu] (-2,-1) circle (1pt);
\draw [fill=xdxdff] (-3,2) circle (1pt);
\draw [fill=xdxdff] (0,-2.093669358048915) circle (1pt);
\draw[color=xdxdff] (0.14923211809477294,-1.7485936939965305) node {$V_1$};
\end{scriptsize}
\end{tikzpicture}
\end{center}
\caption{Illustration for the proof of the Lipschitz continuity of $\varphi_m$ on $K$ (rectangle).
Here $K=S_{m-1}\cap B(0,\sqrt{n}+m-1)$ and $U=K+\frac{w_m}2B^\circ(0,1)$ with the notation of Lemma \ref{lem:local.Lipschitz}.}
\label{fig:Lipschitz}
\end{figure}
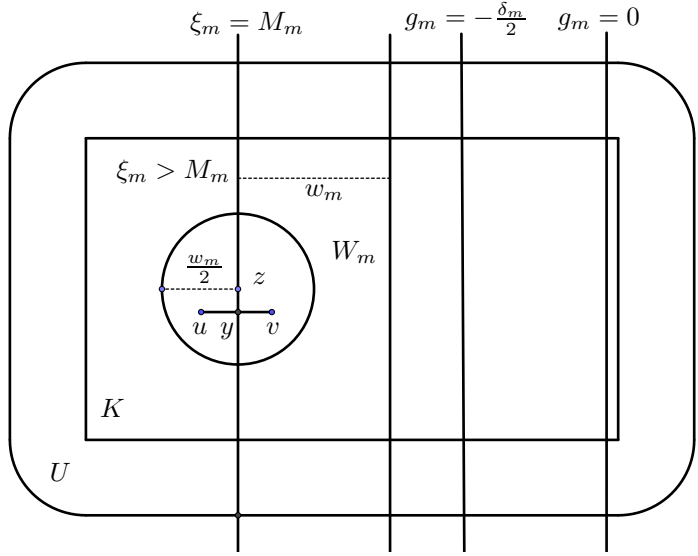


From the proof of Lemma \ref{lem:local.Lipschitz}, the Lipschitz constant of $ \varphi_m$ on $K$ is given by
\begin{equation}
     L_{\bar \varphi_m}:=  \max\left\{\frac{4{C_{\varphi_m}}}{w_m}, L_{\sqrt{\xi_m}}\right\} \,,
\end{equation}
Here we have covered $K$ by a finite sequence of balls with radii $\frac{w_m}{2}$ and centers lying on $K$.
\paragraph{The function $\varphi_m$ has a Lipschitz continuous extension $\bar \varphi_m$.}
Let $\bar \varphi_m:\R^n\to \R$ be the function defined by
\begin{equation}\label{eq:def.zeta}
    \bar \varphi_m(x):=\inf_y\{ \varphi_m(y) + L_{\bar \varphi_m}\|x-y\|_2\,:\,y\in S_{m-1}\cap  B(0,\sqrt{n}+m-1)\}\,.
\end{equation}
By Kirszbraun's theorem (stated in Lemma \ref{lem:Kirszbraun}), $\bar \varphi_m$ is Lipschitz continuous with Lipschitz constant $L_{\bar \varphi_m}$ and $\bar \varphi_m=\varphi_m$ on $S_{m-1}\cap  B(0,\sqrt{n}+m-1)$.

\paragraph{The function $\bar \varphi_m$ is even, positive and has a finite upper bound on $ B(0,\sqrt{n}+m)$ depending on $\varepsilon$.}
Let us prove that $\bar \varphi_m$ is even. 
Consider
\begin{equation}\label{eq:zeta.x}
    \bar \varphi_m(-x)=\inf_y\{ \varphi_m(y) + L_{\bar \varphi_m}\|-x-y\|_2\,:\,y\in S_{m-1}\cap  B(0,\sqrt{n}+m-1) \}\,.
\end{equation}
Let $y$ be any feasible solution of \eqref{eq:zeta.x}.
Since $g_1,\dots,g_{m-1}$ are even, $S_{m-1}\cap  B(0,\sqrt{n}+m)$ is symmetric, i.e., $S_{m-1}\cap  B(0,\sqrt{n}+m)=-S_{m-1}\cap  B(0,\sqrt{n}+m)$, it turns out that $-y$ is a feasible solution of \eqref{eq:zeta.x}.
Thus, 
\begin{equation}
\begin{array}{rl}
    \bar \varphi_m(-x)
    =&\inf_{-y}\{\varphi_m(-y) + L_{\bar \varphi_m}\|-x+y\|_2\,:\,-y\in S_{m-1}\cap  B(0,\sqrt{n}+m-1)\}   \\
      =& \inf_y\{\varphi_m(y) + L_{\bar\varphi_m}\|y-x\|_2\,:\,y\in S_{m-1}\cap  B(0,\sqrt{n}+m-1) \}
     =\bar \varphi_m(x) \,,
\end{array}
\end{equation}
where the latter inequality is due to the fact that $\varphi_m$ is even (since $\xi_m, g_m$ are even).
From this, $\bar \varphi_m$ is even.
It is not hard to show that $\bar \varphi_m\ge \sqrt{M_m}$ since $ \varphi_m\ge \sqrt{M_m}$. 

Let us estimate the upper bound of $\bar \varphi_m$ on $ B(0,\sqrt{n}+m)$. 
Let $x\in  B(0,\sqrt{n}+m)$ and $y\in S_{m-1}\cap  B(0,\sqrt{n}+m-1)$.
From \eqref{eq:def.zeta}, we get
\begin{equation}
    \bar \varphi_m(x) \le \varphi_m(y) + L_{\bar\varphi_m}\|x-y\|_2   \le C_{\varphi_m} + {(2 (\sqrt{n}+m) - 1)} L_{\bar \varphi_m}=:C_{\bar \varphi_m}\,.
\end{equation}
Thus, 
\begin{equation}
\begin{array}{l}
\sup_{x\in B(0,\sqrt{n}+m)}\bar \varphi_m(x)  \le C_{\bar \varphi_m}\,.
\end{array}
\end{equation}
Set $f_{m-1}:=f+\frac{\varepsilon}{2}-\bar\varphi_m^2g_m$.

From \eqref{eq:firstterm} and since $\bar\varphi_m=\varphi_m=\sqrt{\psi_m}$ on $ S_{m-1}\cap B(0,\sqrt{n}+m-1)$, $f_{m-1}\ge 0$ on $S_{m-1}\cap  B(0,\sqrt{n}+m-1)$.
Since $\bar\varphi_m$ is Lipschitz continuous, $f_{m-1}$ is Lipschitz continuous on $ B(0,\sqrt{n}+m)$.

\paragraph{A bound and a Lipschitz constant  of $f_{m-1}$ on $ B(0,\sqrt{n}+m)$ both depend on $\varepsilon$.}
Let us compute an upper bound of $|f_{m-1}|$ on $ B(0,\sqrt{n}+m)$.
Let $y\in B(0,\sqrt{n}+m)$. Then
\begin{equation}
\begin{array}{rl}
     |f_{m-1}(y)| \le |f(y)| +\frac{\varepsilon}{2} + \bar\varphi_m(y)^2|g_m(y)|  \le C_f +  \frac{\varepsilon}{2} + C_{g_m} C_{\bar\varphi_m}^2=:C_{f_{m-1}}\,.
\end{array}
\end{equation}
Thus, 
\begin{equation}\label{eq:bound.f.sphere}
     \|f_{m-1}\|_{ B(0,\sqrt{n}+m)}  \le C_{f_{m-1}}\,.
\end{equation}

We now estimate the Lipschitz constant  of $f_{m-1}$ on $ B(0,\sqrt{n}+m)$.
Let $y,z\in  B(0,\sqrt{n}+m)$ such that $y\ne z$. Then
\begin{equation}
    \begin{array}{rl}
       &\frac{|f_{m-1}(y)-f_{m-1}(z)|}{\|y-z\|_2}\\
       \le & \frac{|f(y)-f(z)|+|\bar\varphi_m(y)^2g_m(y)-\bar\varphi_m(z)^2g_m(z)|}{\|y-z\|_2} \\
        \le &  L_f+\frac{ |\bar\varphi_m(y)^2g_m(y)-\bar\varphi_m(z)^2g_m(y)|}{\|y-z\|_2}\\
        &+\frac{|\bar\varphi_m(z)^2g_m(y)-\bar\varphi_m(z)^2g_m(z)|}{\|y-z\|_2}\\
         =&L_f+\frac{
         |g_m(y)||\bar\varphi_m(y)+\bar\varphi_m(z)||\bar\varphi_m(y)-\bar\varphi_m(z)|+\bar\varphi_m(z)^2|g_m(y)-g_m(z)|}{\|y-z\|_2}\\
         \le&L_f+\frac{ 2C_{g_m}C_{\bar\varphi_m} L_{\bar\varphi_m}\|y-z\|_2+C_{\bar\varphi_m}^2 L_{g_m}\|y-z\|_2}{\|y-z\|_2}\\
         =& L_f+2C_{g_m}L_{\bar\varphi_m}C_{\bar\varphi_m} +L_{g_m}C_{\bar\varphi_m}^2=:L_{f_{m-1}} \,.
    \end{array}
\end{equation}
Then, $L_{f_{m-1}}$ is a Lipschitz constant of $f_{m-1}$ on $ B(0,\sqrt{n}+m)$. 

Notice that $C_{\bar\varphi_{m}}, L_{\bar\varphi_{m}}, C_{f_{m-1}}, L_{f_{m-1}}$ are obtained by composing finitely many times the following operators: ``$+$", ``$-$", ``$\times$", ``$\div$", ``$|\cdot|$",``$(x_1,x_2)\mapsto\max\{x_1,x_2\}$", ``$(x_1,x_2)\mapsto\min\{x_1,x_2\}$", ``$(\cdot)^{\alpha_m}$" and ``$\sqrt{\cdot}$", where all arguments possibly depend on $\varepsilon$.
Without loss of generality we can assume $C_{\bar\varphi_{m}}=\bar r_m\varepsilon^{-r_m}$, $L_{\bar\varphi_{m}}=\bar t_m\varepsilon^{-t_m}$, $C_{f_{m-1}}=\bar c_{m-1}\varepsilon^{-c_{m-1}}$, $L_{f_{m-1}}=\bar l_{m-1}\varepsilon^{-l_{m-1}}$ for some $\bar r_m$, $r_m$, $\bar t_{m}$, $t_m$, $\bar c_{m-1}$, $c_{m-1}$, $\bar l_{m-1}$, $l_{m-1}$ large enough and independent of $\varepsilon$.

\paragraph{Backward induction.}
Repeating the above process (after replacing $f_j$ by $f_{j-1}$) several times, we obtain functions  $\bar\varphi_{j}:\R^n\to\R$, $j=m,m-1,\dots,1$, such that,
\begin{enumerate}
    \item $\bar\varphi_{j}$ is positive, even and bounded from above by  $C_{\bar\varphi_{j}}=\bar r_j\varepsilon^{-r_j}$  on $ B(0,\sqrt{n}+j)$ for some positive constants $\bar r_j$ and $r_j$ independent of $\varepsilon$.
    \item $\bar\varphi_{j}$ is Lipschitz with Lipschitz constant $L_{\bar\varphi_{j}}=\bar t_j\varepsilon^{-t_j}$ for some positive constants $\bar t_j$ and $t_j$ independent of $\varepsilon$.
    \item 
       $f_{j-1}:=f_{j}+\frac{\varepsilon}{2^{m-j+1}}-\bar\varphi_{j}^2g_{j}$ satisfies:
       \begin{enumerate}
           \item $f_{j-1}\ge 0$ on $S_{j-1}\cap  B(0,\sqrt{n}+j-1)$;
           \item  $f_{j-1}\le C_{f_{j-1}}$ on $ B(0,\sqrt{n}+{j})$, where $C_{f_{j-1}}=\bar c_{j-1}\varepsilon^{-c_{j-1}}$  for some positive constants $\bar c_{j-1}$ and $c_{j-1}$ independent of $\varepsilon$;
           \item $f_{j-1}$ is Lipschitz  on $ B(0,\sqrt{n}+{j})$ with Lipschitz constant $L_{f_{j-1}}=\bar l_{j-1}\varepsilon^{-l_{j-1}}$ for some positive constants $\bar l_{j-1}$ and $l_{j-1}$ independent of $\varepsilon$.
       \end{enumerate}
\end{enumerate}
Then 
\begin{equation}
    \begin{array}{rl}
         f_0&=f_1 +\frac{\varepsilon}{2^{m}}-\bar\varphi_{1}^2g_{1} \\
         & = \left(f_2+\frac{\varepsilon}{2^{m-1}}-\bar\varphi_{2}^2g_{2}\right) +\frac{\varepsilon}{2^{m}}-\bar\varphi_{1}^2g_{1}\\
         & = f_2+\left(\frac{\varepsilon}{2^{m-1}}+\frac{\varepsilon}{2^{m}}\right) -\bar\varphi_{2}^2g_{2}-\bar\varphi_{1}^2g_{1}\\
         &=\dots=f_m+\varepsilon\sum_{i=1}^m \frac{1}{2^i} - \sum_{i=1}^m \bar\varphi_{i}^2g_{i}\\
         &=f+\frac{\varepsilon}{2}\frac{1-\frac{1}{2^{m}}}{1-\frac{1}{2}} - \sum_{i=1}^m \bar\varphi_{i}^2g_{i}\\
         &=f+\varepsilon(1-\frac{1}{2^{m}}) - \sum_{i=1}^m \bar\varphi_{i}^2g_{i}\,.
    \end{array}
\end{equation}
From this and since $f_0\ge 0$ on $S_0\cap B(0,\sqrt{n})=B(0,\sqrt{n})\supset [-1,1]^n$, we obtain
\begin{equation}\label{eq:nonega}
    f+\varepsilon - \sum_{i=1}^m \bar\varphi_{i}^2g_{i}\ge \frac{\varepsilon}{2^{m}}\text{ on }[-1,1]^n\,.
\end{equation}
\subsubsection{Polynomial approximations for the weight functions}
\paragraph{Approximating with Bernstein polynomials.}
For each $i\in[m]$, we now approximate $\bar \varphi_i$ on $[-1,1]^n$ with the following Bernstein polynomials: 
\begin{equation}
    B_i^{(d)}(x)=B_{y\mapsto \bar \varphi_i(2y-e),de}\left(\frac{x+e}{2}\right)\,,\quad d\in\N\,,
\end{equation}
with $e=(1,\dots,1)\in\R^n$.
By using Lemma \ref{lem:app.poly}, for all $x\in [-1,1]^n$, for $i\in[m]$,
\begin{equation}
    |B_i^{(d)}(x) - \bar \varphi_i(x) | \le
{L_{\bar \varphi_i}} \biggl(\frac{n}{d} \biggr)^{\frac{1}2}\,,\quad d\in\N\,,
\end{equation}
 and the following inequality holds for all $x\in [-1,1]^n$, for $i\in[m]$:
\begin{equation}
\begin{array}{l}
    |B_i^{(d)}(x)|\le  \sup_{x\in[-1,1]^n}|\bar\varphi_i(x)|\le C_{\bar\varphi_i}\,.
\end{array}
\end{equation}


For $i\in[m]$, let 
\begin{equation}
    d_i:=2u_i\quad\text{with}\quad u_i=\Bigl\lceil{  \frac{2C_{g_i}^2C_{\bar \varphi_i}^2n L_{\bar \varphi_i}^2 (m+1)^22^{2m}}{\varepsilon^2}} \Bigr\rceil\,,
\end{equation}
where $C_{g_i}:=\|g_i\|_{B(0,\sqrt{n}+i)}$, for $i\in[m]$.
Then for all $x\in[-1,1]^n$, 
\begin{equation}
\begin{array}{rl}
    |B_i^{(d_i)}(x) - \bar \varphi_i(x) |  & \le
{L_{\bar \varphi_i}} \left(\frac{n}{d_i} \right)^{\frac{1}2}   \\
&\le
{L_{\bar \varphi_i}} \left(\frac{n}{ \frac{4C_{g_i}^2C_{\bar \varphi_i}^2n L_{\bar \varphi_i}^2 (m+1)^22^{2m}}{\varepsilon^2}} \right)^{\frac{1}2}\\
& = \frac{\varepsilon}{2C_{g_i}C_{\bar \varphi_i}(m+1)2^m}\,.
\end{array}
\end{equation}
\paragraph{Converting to homogeneous approximations.}
For $i\in[m]$, we write 
$B_i^{(d_i)}=\sum_{j=0}^{nd_i}{h_i^{(j)}}$ such that $h_i^{(j)}$ is a homogeneous polynomial with $\deg(h_i^{(j)})=j$.
Set $p_i:=\frac{1}{2}[B_i^{(d_i)}(x)+B_i^{(d_i)}(-x)]$, for $i\in[m]$.
Then $p_i=\sum_{t=0}^{nu_i}{h_i^{(2t)}}$, for $i\in[m]$, since $h_i^{(j)}(x)=h_i^{(j)}(-x)$ if $j$ is even and $h_i^{(j)}(x)=-h_i^{(j)}(-x)$ otherwise.
Since $\bar \varphi_i$ is even, $\bar \varphi_i(x)=\frac{1}{2}[\bar \varphi_i(x)+ \bar \varphi_i(-x)]$.
It implies that for $x\in[-1,1]^n$, for $i\in[m]$,
\begin{equation}
\begin{array}{rl}
    |p_i(x) - \bar \varphi_i(x)| & 
    =| \frac{1}{2}[B_i^{(d_i)}(x)+B_i^{(d_i)}(-x)] - \frac{1}{2}[\bar \varphi_i(x)+\bar \varphi_i(-x)]| \\
     & \le \frac{1}{2}|B_i^{(d_i)}(x) - \bar \varphi_i(x) |+\frac{1}{2}|B_i^{(d_i)}(-x) - \bar \varphi_i(-x) |\\ 
     &\le
\frac{\varepsilon}{4C_{g_i}C_{\bar \phi_i}(m+1)2^m}+\frac{\varepsilon}{4C_{g_i}C_{\bar \varphi_i}(m+1)2^m}=\frac{\varepsilon}{2C_{g_i}C_{\bar \varphi_i}(m+1)2^m}\,.
\end{array}
\end{equation}
and 
\begin{equation}
    |p_i(x)|\le  \frac{1}{2}(|B_i^{(d_i)}(x)|+|B_i^{(d_i)}(x)|) \le \frac{1}{2}(C_{\bar \varphi_i}+C_{\bar \varphi_i})=C_{\bar \varphi_i}\,.
\end{equation}
Set $q_i:=\sum_{t=0}^{nu_i}{h_i^{(2t)}}\|x\|_2^{2(nu_i-t)}$.
Then $q_i$ is a homogeneous polynomial of degree $2nu_i$ and $q_i=p_i$ on $\mathbb{S}^{n-1}$, for $i\in[m]$. 
Thus for $i\in[m]$, $|q_i(x) - \bar \varphi_i(x)|\le \frac{\varepsilon}{2C_{g_i}C_{\bar \varphi_i}(m+1)2^m}$ and $|q_i(x)|\le C_{\bar \varphi_i}$, for all $x\in \mathbb{S}^{n-1}$.
From these and \eqref{eq:nonega}, for all $x\in \mathbb{S}^{n-1}$, \begin{equation}
    \begin{array}{rl}
     & f(x)+\varepsilon - \sum_{i=1}^m q_{i}(x)^2g_{i}(x)\\
       =&f(x)+\varepsilon - \sum_{i=1}^m \bar\varphi_{i}(x)^2g_{i}(x) + \sum_{i=1}^m g_{i}(x)[\bar\varphi_{i}(x)^2-q_{i}(x)^2]\\
          \ge& \frac{\varepsilon}{2^{m}} - \sum_{i=1}^m |g_{i}(x)||\bar\varphi_{i}(x)+q_{i}(x)||\bar\varphi_{i}(x)-q_{i}(x)|\\
         \ge& \frac{\varepsilon}{2^{m}} - \sum_{i=1}^m C_{g_{i}(x)}(|\bar\varphi_{i}(x)|+|q_{i}(x)|)\frac{\varepsilon}{2C_{g_i}C_{\bar \varphi_i}(m+1)2^m}\\
         \ge &\frac{\varepsilon}{2^{m}} - \sum_{i=1}^m 2C_{g_{i}}C_{\bar\varphi_{i}}\frac{\varepsilon}{2C_{g_i}C_{\bar \varphi_i}(m+1)2^m}\\
         =& \frac{\varepsilon}{2^{m}} - \frac{m\varepsilon}{(m+1)2^m}=\frac{\varepsilon}{(m+1)2^m}\,.\\
    \end{array}
\end{equation}
Moreover, for all $x\in \mathbb{S}^{n-1}$, \begin{equation}
    \begin{array}{l}
     f(x)+\varepsilon - \sum_{i=1}^m q_{i}(x)^2g_{i}(x)
       \le C_f+\varepsilon + \sum_{i=1}^m C_{\bar \varphi_i}^2C_{g_{i}}=:C_F\,.
    \end{array}
\end{equation}

\subsubsection{Applying the global positivity certificate}
Set $D:=\max_{i\in[m]}\{2nu_i+d_{g_i},\df\}$ and 
\begin{equation}
\begin{array}{l}
F=\|x\|_2^{2(D-\df)}(f+\varepsilon\|x\|_2^{2\df}) - \sum_{i=1}^m g_{i} q_{i}^2\|x\|_2^{2(D-2nu_i-d_{g_i})}\,.
\end{array}
\end{equation}
Then $F$ is a homogeneous polynomial of degree $2D$ and for all $x\in\mathbb{S}^{n-1}$,
\begin{equation}
\begin{array}{l}
C_F\ge F(x)=f(x)+\varepsilon - \sum_{i=1}^m q_{i}(x)^2g_{i}(x)\ge \frac{\varepsilon}{(m+1)2^m}\,. 
\end{array}
\end{equation}
It implies that $F$ is a positive definite form of degree $2D$ with $\inf_{x\in\mathbb{S}^{n-1}}F(x)\ge \frac{\varepsilon}{(m+1)2^m}$ and $\sup_{x\in\mathbb{S}^{n-1}}F(x)\le C_F$.
There is no loss of generality in assuming $C_F=b\varepsilon^{-b}$ for some large enough $b>0$  independent of $\varepsilon$.
Similarly assume that $D\ge d\varepsilon^{-d}$ for some large enough $d>0$  independent of $\varepsilon$.
From this, 
\begin{equation}
    \Theta(F)\le \frac{b\varepsilon^{-b}}{\frac{\varepsilon}{(m+1)2^m}}=b(m+1)2^m\varepsilon^{-b-1}\,.
\end{equation}
Set 
\begin{equation}
    \bar K:= \frac{{2nd\varepsilon^{-d}(2d\varepsilon^{-d} - 1)}}{{4\log 2}}b(m+1)2^m\varepsilon^{-b-1} \,.
\end{equation}
Then
\begin{equation}
    \bar K\ge \frac{{2nD(2D - 1)}}{{4\log 2}}\Theta (F) - \frac{{n + 2D}}{2}\,.
\end{equation}
Clearly there exist positive constants $\bar c$ and $c$  independent of $\varepsilon$ such that $ \bar c\varepsilon^{-c}\ge \bar K$.
Let $K\in\N$ and $K\ge \bar c\varepsilon^{-c}\ge \bar K$.
According to Lemma \ref{lem:homogeneous.sos}, there exists a homogeneous SOS polynomial $s_0$ of degree $2(D+K)$ such that $\|x\|_2^{2K}F=s_0$.
It implies that
\begin{equation}
\begin{array}{rl}
    \|x\|_2^{2(D-\df+K)}(f+\varepsilon\|x\|_2^{2\df})  &  =s_0+ \sum_{i=1}^m g_{i} q_{i}^2\|x\|_2^{2(D-2nu_i-d_{g_i}+K)}\\
     & =s_0+ \sum_{i=1}^m g_{i} s_i\,,
\end{array}
\end{equation}
where $s_i:= q_{i}^2\|x\|_2^{2(D-2nu_i-d_{g_i}+K)}$ is a homogeneous SOS polynomial such that $\deg(g_is_i)=2(K+D)$, for $i\in[m]$.
Set $k=D-\df+K$. 
Then $\|x\|_2^{2k}(f+\varepsilon\|x\|_2^{2\df})=s_0+ \sum_{i=1}^m g_{i} s_i$ with $\deg(s_0)=\deg(g_is_i)=2(k+\df)$, for $i\in[m]$.
\paragraph{The case of the ice cream constraint.}
Assume that $m=1$ and $g_1=x_n^2-\|x'\|_2^2$ with  $x':=(x_1,\dots,x_{n-1})$. 
We shall show that $c=65$.
Using Lemma \ref{lem:lojas.ineq.ball}, we can take $\alpha_m=2$ in \eqref{eq:lojasiewicz.ineq}.
We then obtain the following asymptotic equivalences as $\varepsilon\to 0^+$:
\begin{equation}
    \begin{array}{rl}
       &\delta_m \sim R_1\varepsilon^2 \Rightarrow C_{\psi_m} \sim R_2\varepsilon^{-2} \Rightarrow C_{\varphi_m} \sim R_3\varepsilon^{-1}\Rightarrow w_m \sim R_4\varepsilon^5 \Rightarrow L_{\sqrt{\xi_m}}\sim R_5\varepsilon^{-\frac{9}{2}}\\
       &\Rightarrow L_{\bar \varphi_m} \sim R_6\varepsilon^{-6} \Rightarrow C_{\bar \varphi_m} \sim R_7\varepsilon^{-6} 
       \Rightarrow u_m \sim R_{8}\varepsilon^{-26} \Rightarrow d_m \sim R_{9}\varepsilon^{-26} \\
       &\Rightarrow C_F\sim R_{10}\varepsilon^{-12}\Rightarrow D\sim R_{11}\varepsilon^{-26} 
       \Rightarrow b=12\Rightarrow d=26 \Rightarrow \bar K\sim R_{12}\varepsilon^{-65}\\
       & \Rightarrow c = 65.
    \end{array}
\end{equation}
for some $R_j>0$ independent of $\varepsilon$, $j\in[12]$.
This completes the proof of Theorem \ref{theo:complex.putinar.vasilescu}.

\bibliography{reference} 

\end{document}